\documentclass[11pt]{article}

\usepackage{amsfonts, amssymb, eucal, amsmath, eufrak,amsthm}
\usepackage{graphicx}
\usepackage{srcltx}
\usepackage{amsmath}
\usepackage{amssymb}
\usepackage{makeidx}
\usepackage{subfloat}
\usepackage{subcaption}
\usepackage{algorithmic}
\usepackage{xcolor}
\usepackage{enumerate}
\usepackage{thmtools}
\usepackage{bm}
\newcommand{\R}{\mathbb R}
\newcommand{\Rn}{\mathbb R^n}

\newcommand{\RnN}{\mathbb R^{nN}}

\newcommand{\Rnn}{\mathbb R^{n\times n}}

\newcommand{\N}{\mathbb N}
\declaretheoremstyle[
  bodyfont=\normalfont\itshape,
  headformat=\NAME\NUMBER  
]{}
\declaretheorem[style=nospacetheorem,name=Assumption ]{manualtheoreminner}
\newenvironment{assumption}[1]{%
  \renewcommand\themanualtheoreminner{#1}%
  \manualtheoreminner
}{\endmanualtheoreminner}
\newtheorem{alg}{Algorithm}[section]
\newtheorem{lemma}{Lemma}[section]
\newtheorem{theorem}{Theorem}[section]
\newtheorem{remark}{Remark}[section]
\newtheorem{corollary}{Corollary}[section]

\newcommand{\tr }{^\top}

\DeclareMathOperator*{\dia}{diag}

\newcommand{\calG}{\mathcal G}

\newcommand{\calK}{\mathcal K}
\newcommand{\calE}{\mathcal E}
\newcommand{\calV}{\mathcal V}
\newcommand{\calW}{\mathcal W}

\newcommand{\xbf}{\mathbf x}
\newcommand{\rbf}{\mathbf r}
\newcommand{\zbf}{\mathbf z}

\newcommand{\gbf}{\mathbf g}

\newcommand{\dbf}{\mathbf d}
\newcommand{\ybf}{\mathbf y}

\newcommand{\bbf}{\mathbf b}
\newcommand{\vbf}{\mathbf v}

\newcommand{\lr}[1]{\left(#1\right)}

\newcommand{\be}{\begin{equation}}
\newcommand{\ee}{\end{equation}}

\begin{document}

\title{Distributed Inexact Newton Method with Adaptive Step Sizes}
\author{Du\v{s}an Jakoveti\'c\footnote{Department of Mathematics and Informatics, Faculty of Sciences, University of Novi Sad, Trg Dositeja
Obradovi\'ca 4, 21000 Novi Sad, Serbia, Email: dusan.jakovetic@dmi.uns.ac.rs, natasak@uns.ac.rs, The work of D. Jakoveti\'c and N. Kreji\'c  was supported by the Science Fund of the Republic of Serbia, Grant no. 7359, Project LASCADO.},  Nata\v{s}a Kreji\'c\footnotemark[1],   Greta Malaspina\footnote{Dipartimento di Ingegneria Industriale, Università degli studi di Firenze, Viale G.B. Morgagni 40,
50134 Firenze, Italy. Member of the INdAM Research Group GNCS.
 Email: greta.malaspina@unifi.it. This work is partially supported by the BIGMATH project which has received funding by the European Union’s Horizon 2020 research and innovation programme under the Marie Sk\l{}odowska-Curie Grant Agreement no. 812912, and by Partenariato esteso FAIR ``Future Artificial Intelligence Research'' SPOKE 1 Human-Centered AI. Obiettivo 4, Project ``Mathematical and Physical approaches to innovative Machine Learning technologies (MaPLe)''.
}
}
\date{}

\maketitle
\begin{abstract}
We consider two formulations for distributed optimization wherein $N$ nodes in a generic connected network solve a problem of common interest: distributed personalized optimization  and consensus optimization. A new method termed DINAS (Distributed Inexact Newton method with Adaptive step size) is proposed. DINAS employs large adaptively computed step sizes, requires a reduced global parameters knowledge with respect to existing alternatives, and can operate without any local Hessian inverse calculations nor Hessian communications. When solving personalized distributed learning formulations, DINAS achieves quadratic convergence with respect to computational cost and  linear convergence with respect to communication cost, the latter rate being independent of the local functions condition numbers or of the network topology. When solving consensus optimization problems, DINAS is shown to converge to the global solution. Extensive numerical experiments demonstrate significant improvements of DINAS over existing alternatives.  As a result of independent interest, we provide for the first time convergence analysis of the Newton method with the adaptive Polyak's step size when the Newton direction is computed inexactly in centralized environment.\\

\end{abstract}

\section{Introduction}
We consider distributed optimization problems 
where a network of $N$ computational nodes, interconnected by a generic undirected 
connected network, collaborate in order to solve a problem of common interest. Specifically, 
we consider the following two widely used problem formulations:
\be 
\label{eq:penalty-intro}
\min \sum_{i=1}^{N} f_i(x_i) + \frac{1}{2 \beta} \sum_{\{i,j\} \in \calE} w_{ij} \|x_i-x_j\|^2
\ee
\be  \label{eq:minsum_newton}
\min f(y) = \sum_{i=1}^{N} f_i(y).
\ee
Here, $\|\cdot\|$ denotes the Euclidean norm; $f_i: \mathbb{R}^n \to \mathbb{R}$ is a local cost function available to node $i$; 
$x_i \in {\mathbb R}^n$, $i=1,...,n$, and $y \in {\mathbb R}^n$ are optimization variables; 
$i=1,...,N$; $\calE$ denotes the set of undirected edges; $\beta>0$ is a parameter; and $w_{ij}>0$, $\{i,j\} \in \calE$,
are positive constants.  
 Formulation \eqref{eq:penalty-intro} arises, e.g., in personalized decentralized machine learning, e.g., \cite{DJAM, ribeiro, variable, ErminWei,Bellet}; see also \cite{Richtarik} for a related formulation in federated learning settings. 
 Therein, each node $i$ wants to find its personalized local model $x_i$ by fitting it against its local data-based loss function $f_i$. In typical scenarios, 
 the data available to other nodes $j \neq i$ are also useful to node $i$. One way to enable node $i$ 
 learn from node $j$'s local data is hence to impose a similarity between the different 
 nodes' models. This is done in \eqref{eq:penalty-intro} by adding a quadratic penalty term, 
 where parameter $w_{ij}>0$ encodes the ``strength'' of the similarity between 
 node $i$'s and node $j$'s models.\footnote{In order to simplify the presentation, we 
 let the $w_{ij}$'s satisfy $\sum_{j \neq i} w_{ij}<1$ and also introduce an 
 arbitrary common multiplicative constant $\beta>0$.} The quantities $w_{ij}$'s 
 also respect the sparsity pattern of the underlying network, i.e., they are non-zero only 
 for the node pairs $i,j$ that can communicate directly, i.e., for those pairs 
 $i,j$ such that $\{i,j\} \in \calE.$ Throughout the paper, we refer to formulation \eqref{eq:penalty-intro} as 
 distributed personalized optimization, or penalty formulation. 

Formulation \eqref{eq:minsum_newton} has also been extensively studied, e.g., \cite{DG,dusan,DQNSIAM, cdc-submitted,arxivVersion,harnessing,  diging,extra, lessard,xu}.  
It is related with \eqref{eq:penalty-intro} but, unlike 
\eqref{eq:penalty-intro}, there is a global model (optimization variable) 
$y \in {\mathbb R}^n$ common to all nodes, i.e., the nodes
 want to reach a consensus (agreement) on the model. 
 Formulation \eqref{eq:penalty-intro} may be seen as an inexact penalty variant  
 of \eqref{eq:minsum_newton}, with the increasing approximation accuracy as 
 $\beta$ becomes smaller. This connection between the two formulations 
  has also been exploited in prior works, e.g., \cite{NetworkNewton},  
  in order to develop new methods for solving~\eqref{eq:minsum_newton}. 
  In other words, once a distributed solver for   \eqref{eq:penalty-intro} is available,
  one can also solve~\eqref{eq:minsum_newton} approximately
   by taking a small $\beta$, or exactly by solving a sequence of 
   problems \eqref{eq:minsum_newton} with an appropriately tuned decreasing sequence of 
   $\beta$'s. We refer to formulation \eqref{eq:minsum_newton} as distributed consensus 
   optimization, or simply consensus optimization.
   
In this paper, we also follow this path and first devise a method for solving 
\eqref{eq:penalty-intro} with a fixed $\beta$; we then capitalize on the latter to derive solvers for 
\eqref{eq:minsum_newton}.

There is a vast literature for solving problems of type \eqref{eq:minsum_newton} and 
\eqref{eq:penalty-intro}. We first discuss formulation  \eqref{eq:minsum_newton}. In general one can distinguish between several classes of methods. To start, a distributed gradient (DG) method has been considered in \cite{DG}. 
Given that the gradient method with fixed step size cannot achieve exact convergence, a number of methods, belonging to the class of gradient tracking methods  which assume that additional information is shared among nodes, are proposed and analysed,  \cite{dusan,DQNSIAM, cdc-submitted,arxivVersion,harnessing,  diging,extra, lessard,xu}.   In these methods the objective function is assumed to be   convex with Lipschitz-continuos gradient, the gradient tracking scheme is employed with a fixed steps size, and the exact convergence to is reached. The step size required in order to ensure convergence needs to be smaller than a threshold that depends on global constants such as strong convexity and gradient's Lipschitz constant; the constant also needs to be known  by all nodes beforehand.

 In \cite{diging2,xu,Khan1,Khan2,Sayed1}, the case of uncoordinated time-invariant step sizes is considered. 

In \cite{dsg}, a modification of the gradient tracking method, with step sizes varying both across nodes and across iterations is proposed.  Time varying networks are considered in \cite{diging, Khan3}. Again, the step sizes in all of these methods depend on some global constants (like strong convexity and Lipschitz). The  actual implementation is subject to estimation of these constants and  very often leads to costly hand-tuning of step size parameters. 

If the optimization problem of interest is ill-conditioned, first order methods might be very slow, and second order methods appear as an attractive alternative.  When considering second order methods in a distributed and large-scale framework, there are a few issues that need to be addressed. First, as in the classical (centralized) optimization framework, Hessian-based methods are typically quite expensive: while in the distributed framework the work of computing the Hessian is shared by the nodes and therefore may not represent a major issue, solving the linear system typically required at each iteration  to compute the Newton-like direction may be prohibitively expensive for problems of very large sizes. Secondly, while we normally assume that nodes can share the local vector of variables and the local gradients (or some other vectors of magnitudes comparable to the variable dimension $n$), sharing the Hessian cause communication costs quadratic in $n$. 
Finally,  one would like  to have superlinear or quadratic convergence for a second order method. However, in a distributed framework when solving \eqref{eq:minsum_newton} or \eqref{eq:penalty-intro}, there are two main obstacles: reaching consensus, and the choice of the step size. 
These issues are addressed in the literature in several ways so far. 

We first discuss second order methods that solve \eqref{eq:penalty-intro}. In \cite{NetworkNewton} the authors propose a distributed version of the classical Newton method, dubbed Network Newton -- NN, that relies on a truncated Taylor expansion of the inverse Hessian matrix. 
NN solves exactly \eqref{eq:penalty-intro} and yields an inexact solution of \eqref{eq:minsum_newton}.  The approximation of the inverse Hessian with NN   involves inversion of local (possibly dense) Hessians of the $f_i$'s. 
Linear convergence of the method to the solution of \eqref{eq:penalty-intro} is proved if standard assumptions for the Newton method hold, for a suitable choice of the step size.
 In \cite{RNN-k}  a modification of NN method is proposed, that employs a reinforcement strategy.

Several second order methods are proposed that solve~\eqref{eq:minsum_newton}. In \cite{primdual} the direction is computed using the same approximation strategy as in NN, but within primal-dual strategy. The method proposed in  \cite{JKK} relies on primal-dual framework as well but computes the direction inexactly.  
Thus, the method avoids any kind of matrix inversion and relies on a fixed point iterative method for computing the search direction. This leads to favorable performance for problems with a relatively large $ n. $
In \cite{Ntrack} the gradient-tracking strategy is extended to the case of the Newton method.

Current second order methods that solve \eqref{eq:penalty-intro} such as \cite{NetworkNewton} and \cite{ErminWei} 
 converge at best linearly to the solution of \eqref{eq:penalty-intro}.\footnote{It has been shown that the methods in \cite{NetworkNewton} and \cite{ErminWei} exhibit an error decay of a quadratic order only for a 
 bounded number of iterations, while the actual convergence rate is linear.} Moreover, 
 the step size by which a Newton-like step is 
 taken depends on several global constants and is usually very small, leading to slow convergence in applications.   
 The choice of the step size is of fundamental importance  for a method to achieve both global convergence and fast local convergence. Classical line-search strategies, which effectively solve this issue in the centralized framework, are not applicable in the distributed setting as they require knowledge of the value of the global function at all nodes and may require several evaluation at each iteration, which would yield prohibitively large communication traffic. In \cite{flooding} the authors propose a method termed DAN,  
 \cite{flooding} proposes a method for solving \eqref{eq:minsum_newton}, but it may be generalized to 
solving \eqref{eq:penalty-intro} such that similar convergence rate guarantees hold order-wise. The method employs a finite-time communication strategy to share the local hessians
through the whole network and makes use of the step size proposed in \cite{Polyak} for the Newton method in the centralized case, which ensures both global convergence and fast local convergence. 
Given that sharing all local hessians among nodes might be too expensive for problems of even moderate size, a version of DAN that employs and shares through the whole network rank one approximations of the local hessians is proposed.  The step size is again defined by the global constants such as strong convexity and Lipschitz constants. In \cite{daneshmand} the authors propose a distributed Newton method with cubic regularization for solving empirical risk minimization problems. 
In \cite{Berglund} the authors propose a Distributed Newton method for the solution of quadratic problems that requires the nodes to share both vectors and matrices, and prove that the convergence on the method depends only on the logarithm of the condition number of the problem. In particular they prove that sharing second order information can improve the performance of distributed methods when the condition number of the problem is large. In \cite{liu2022} the authors combine the idea of sharing compressed Hessian information with that of performing several rounds of consensus at each iteration of the method, then show that the method they propose achieves asymptotically superlinear convergence when the number of consensus rounds at every iterations tends to infinity.\\

In summary, current second order methods for solving \eqref{eq:penalty-intro} suffer from at least one of the following drawbacks. Only linear convergence rate is achieved, both with respect to the number of local functions $f_i$'s gradient and Hessian evaluations, and with respect to the number of inter-node communication rounds. The methods require knowledge of global constants beforehand in order for step size to be set, usually leading to small step size choices and slow convergence. If a superlinear or quadratic convergence is reached, this comes at a $O(n^2)$ cost of local Hessians' communication.\\

\textbf{Contributions}. We propose a new method termed DINAS (Distributed Inexact Newton method with Adaptive step size) 
that can solve both \eqref{eq:penalty-intro} and \eqref{eq:minsum_newton} formulations. 
 When solving~\eqref{eq:penalty-intro}, DINAS achieves linear, superlinear  or quadratic convergence to the solution in terms of the number of outer iterations. In more detail, 
the method achieves, in a best scenario, a quadratic convergence with respect to the number of local gradient and local Hessian evaluations, and a linear convergence with respect to the number of communication rounds. Moreover, the communication-wise linear convergence rate is independent of local functions' condition numbers or of the underlying network topology, and it is hence significantly improved over the linear rates achieved in the literature so far. DINAS uses an adaptive step size that can be computed at each iteration with a reduced knowledge of global constants required when compared with existing alternatives. It computes the search direction by using a suitable linear solver, i.e., an approximate solution  of the Newton linear system is computed in a distributed way at each iteration via an iterative solver. To make the presentation easier to follow we specify here the method with Jacobi Overrelaxation, JOR,  as the linear solver of choice, but any other linear solver that is implementable in the distributed environment can be applied. After computation of the approximate Newton direction of suitable level of inexactness, the step size is computed based on the progress in gradient decrease, and the new iteration is computed. The step size  computation is an adaptation to the distributed framework of the step size presented in \cite{Polyak} for the classical Newton method. This way we achieve global convergence of the method. Furthermore, adaptive step size allows us to maintain fast local convergence. The rate of convergence is controlled by the forcing terms as in the classical Inexact Newton methods, \cite{INmethod}, depending on the quality of approximation of the Newton direction obtained by JOR method. Thus one can balance the cost and convergence rate depending on the desired precision and adjust the forcing term during the iterative process. 
 The fact that we do not need any matrix inversion and solve the linear system only approximately reduces the computational cost in practice by several orders of magnitude, in particular if the dimension $ n $ is relatively large.  

For formulation \eqref{eq:minsum_newton}, we employ DINAS by solving a sequence of problems \eqref{eq:penalty-intro} with an appropriately chosen decreasing sequence of penalty parameters $\beta$, and show that the overall scheme converges to the solution of \eqref{eq:minsum_newton}. 
While we do not have theoretical convergence rate guarantees for solving \eqref{eq:minsum_newton}, extensive numerical experiments demonstrate significant improvements over state-of-the-art distributed second order solvers of 
 \eqref{eq:minsum_newton}. The results also confirm that DINAS is particularly efficient for large dimensional (large $ n$) settings. 

Finally, an important byproduct of the convergence theory for the distributed case developed here is the convergence of the classical (centralized)  Newton method with the Polyak adaptive step size when the Newton directions are computed inexactly. To the best of our knowledge, only the exact Newton method with the Polyak step size has been analyzed to date.

This paper  is organized as follows. 
In Section 2 we describe the distributed optimization framework that we consider and formulate the  assumptions. The DINAS algorithm is presented and discussed  in Section 3. The analysis of DINAS for formulation \eqref{eq:penalty-intro} is given in Section 4.  
 The analysis of the centralized Newton method with the Polyak step size and inexact Newton-like updates is 
 given in Section 5, while the analysis for formulation \eqref{eq:minsum_newton} 
is presented in Section 6.  
Numerical results are provided in Section 7. We offer some conclusions in  Section 8.

\section{Model and preliminaries}

Assume that a network $\calG = \left(\calV,\calE\right)$ of computational nodes is such that node $i$ holds the function $f_i$ and can communicate directly with all its neighbors in $\mathcal G$. Moreover, assume that each node hold a local vector of variables $\xbf_i\in\Rn. $ The properties of the communication network are stated in the assumption below. 
\begin{assumption}{A1}\label{ass:network} The network $\calG = (\calV,\calE)$ is undirected, simple and connected, and it has self-loops at every node, i.e.,  $(i,i)\in\calE$ for every $i=1,\dots,N$.
\end{assumption}

We associate the communication matrix $ W $ to the graph $ \calG $ as follows.

\begin{assumption}{A2}\label{ass:consensus} The matrix $W=[w_{ij}] \in\R^{N\times N}$ is symmetric and doubly stochastic such that if 
 $(i,j)\notin\calE$ then $w_{ij}=0, w_{ij} \neq 0, (i,j) \in\calE $  and $ w_{ii} \neq 0, i=1,\ldots,N. $
\end{assumption}
Double stochastic matrix means that all elements of the matrix are nonnegative and they sum up to 1 by columns and rows, i.e. $ \sum_{i=1}^n w_{ij} =1, \sum_{j=1}^{n} w_{ij} =1.  $ Hence $ w_{ij} \in [0,1]. $
Given the communication matrix $ W$ that satisfies Assumption \ref{ass:consensus} we can associate each node $ i $ with its neighbors $ j \in O_i.  $ Clearly $ w_{ij} \neq 0 $ if $ j \in O_i,  $ $ w_{ii} \neq 0 $ and  $ \sum_{j \in O_i} w_{ij}+ w_{ii} = 1. $ The communication matrix is also called consensus matrix.  

The method we propose here is a Newton-type method, hence we assume the standard properties of the objective function. We use the symbol $ \prec $ to denote the following matrix property: $ A \prec B $ means that $ A-B $ is a positive definite matrix. 
\begin{assumption}{A3} \label{ass:objetive}  Let each $ f_i,  i=1,\ldots, N $ be a two times continuously differentiable, strongly  convex function such that for some $ 0 < \mu < M  $ there holds
\be \label{f:property1}
\mu I \prec \nabla^2 f_i(y) \prec M I.  
\ee
The Hessian is Lipschitz continuous and for some $ L > 0 $ and all $ y, z \in \Rn $ we have
\be \label{f:property2}
\|\nabla^2 f_i(y) - \nabla^2 f_i(z)\|_{\infty} \leq L \|y-z\|_{\infty}.
\ee
\end{assumption}

We are also interested in  \eqref{eq:penalty-intro}.  
It can be expressed  as follows. Let
$$ F(\xbf) = \sum_{i=1}^N f_i(\xbf_i), $$
where $ \xbf=(\xbf_1,\ldots,\xbf_N) \in \RnN, \; \xbf_i \in \Rn. $ 
Then (\ref{eq:minsum_newton}) is equivalent to the following constrained problem 
$$ \min_{\xbf \in \RnN}  F(\xbf) \mbox{ s.t. } \xbf_i = \xbf_j, \; i,j \in \{1,\ldots, N\}. $$
Notice that Assumption \ref{ass:objetive} implies that $ F $ is strongly convex and its Hessian is Lipschitz continuous with the same constants $ \mu, M $ and $ L. $ 

Assuming that $W$ is a consensus matrix associated with the network $\calG$, we define $\calW = W\otimes I_n$ with $ I_n $ being the identity matrix in $ \Rn $ and $ \otimes $ the Kronecker product. Since $W$ is a doubly stochastic matrix, we have that $\calW\xbf = \xbf$ if and only if $\xbf_i = \xbf_j$ for every $i,j=1,\dots, N, $ \cite{NetworkNewton} and therefore, problem \eqref{eq:minsum_newton} is equivalent to  the following problem
\begin{equation}\label{eq:minsum_constrW}
    \min_{\xbf\in\RnN} F(\xbf),\ \text{s.t.}\ (I-\calW)^{1/2}\xbf=0. 
\end{equation}
Clearly, one can state the constraint in \eqref{eq:minsum_constrW} in different ways, for example $ (I-\calW)\xbf=0 $, but the formulation in \eqref{eq:minsum_constrW} is the most common one as it allows us to consider the quadratic penalty as stated in \eqref{eq:penaltyref} below. Given $\beta>0$, we consider the following problem, equivalent to \eqref{eq:penalty-intro}:
\begin{equation}\label{eq:penaltyref}
    \min_{\xbf\in\RnN}\Phi_\beta(\xbf)\ \text{with}\ \Phi_\beta(\xbf)= F(\xbf)+\frac{1}{2\beta}\xbf\tr(I-\calW)\xbf.
\end{equation}
The relationship between \eqref{eq:penaltyref} and \eqref{eq:minsum_newton} is analysed in \cite{penaltyref}. In particular, it is known that the penalty problem \eqref{eq:penaltyref} yields an approximate solution of \eqref{eq:minsum_newton} such that $ \|\xbf^*_i - \ybf^*\| = {\cal O}(\beta/(1-\lambda_2)) $ where $ \xbf^*=(\xbf_1^*,\ldots,\xbf_N^*) $ is the solution of (\ref{eq:penaltyref}), $ \ybf^* $ is the solution of \eqref{eq:minsum_newton} and $ \lambda_2 $ is the second largest eigenvalue of $ W. $  First  we will be concerned with the solution of (\ref{eq:penaltyref}). Later in Section 6 we deal with a sequence of penalty problems, with decreasing values of the penalty parameter $ \beta $ that allow us to reach the solution of (\ref{eq:minsum_newton}) with an arbitrary precision. This approach in some sense mimics the penalty decomposition methods in the classical computational framework, where one deals with a sequence of penalty problems, defined in different ways, to solve the original problems. The decomposition is usually defined splitting the set of constraints into sets of easy and difficult constraints, and then a sequence of problems is solved, possibly inexactly, using different optimality conditions, for example see \cite{KL2}.

\section{Algorithm DINAS: Personalized distributed optimization}

The classical Inexact Newton iteration for (\ref{eq:penaltyref}), given the current iteration $ \xbf^k $ and the forcing parameter $ \eta_k, $ is defined as 
$$ \xbf^{k+1} = \xbf^k + \alpha_k \dbf^k, $$
where $ \alpha_k > 0 $ is the step size, $\dbf^k $ is computed from 
\be \label{eq:Nsystem} 
 \nabla^2 \Phi_{\beta}(\xbf^k) \dbf^k = \nabla \Phi_{\beta}(\xbf^k) + \rbf^k
 \ee 
 and $ \rbf^k $ is the residual vector that satisfies 
\be \label{eq:residual} 
\|\rbf^k\|_{\infty} \leq \eta_k  \| \nabla \Phi_{\beta}(\xbf^k)\|_{\infty}. 
\ee
 The forcing term $ \eta_k $ is bounded from above, $ \eta_k \leq \eta < 1 $ while the step size $ \alpha_k $ is determined by line search or some other globalization strategy, \cite{EW}. If $ \eta_k = 0 $ the method is in fact the Newton's method and the step size can be determined as in \cite{Polyak}.  Notice that we are using the norm $ \|\cdot\|_{\infty} $ in the above expressions as that norm is suitable for the distributed case we are interested in.

In the context of the setting presented in \eqref{eq:Nsystem}--\eqref{eq:residual}, we can distinguish our theoretical and numerical contributions. Our theoretical contributions are grounded on 1) novel analysis of inexact Newton methods when the residual error satisfies condition~\eqref{eq:residual}; and 2) showing that a wide class of iterative distributed solvers can be designed such that it satisfies condition~\eqref{eq:residual}. Our numerical contributions are to develop, test and validate  efficient iterative solvers that solve \eqref{eq:Nsystem} up to accuracy~\eqref{eq:residual}.

 To apply an Inexact Newton method in distributed framework we need to compute the direction $ \dbf^k $ such that (\ref{eq:residual}) holds and to determine the step size. 
 Note that directly solving \eqref{eq:Nsystem} would require the inversion of matrix $\nabla^2 \Phi_{\beta}(\xbf^k)$. This is however not amenable in distributed setting. Namely, while the matrix $\nabla^2 \Phi_{\beta}(\xbf^k)$ has a block-sparsity structure that respects the sparsity pattern of the network, its inverse may not have this property, and hence a direct solution of \eqref{eq:Nsystem} by the matrix inversion is not possible. However, as the matrix $\nabla^2 \Phi_{\beta}(\xbf^k)$ is block-sparse according to the network structure, we can apply carefully designed iterative fixed point solvers in order to solve \eqref{eq:Nsystem} up to the accuracy defined by (\ref{eq:residual}). 
   For this purpose, a number of methods for solving linear systems in distributed framework may be used, \cite{dls0,nedic, DFIX}. It turns out that a class of fixed point methods may be used  without any   changes, \cite{EFIX,JKK}.  
 The application of Jacobi Overrelaxation (JOR) method to the system (\ref{eq:Nsystem}) is specified here but the theoretical analysis holds for any solver assuming that (\ref{eq:residual}) is valid.

 For the sake of clarity let us briefly state the JOR method for solving (\ref{eq:Nsystem}), while further details can be seen in \cite{EFIX, JKK}. First of all notice that the system matrix $ \nabla^2 \Phi_{\beta}(\xbf^k)= \nabla^2 F(\xbf^k) + \frac{1}{\beta }(I - \calW)$ is symmetric and positive definite due to strong convexity of $ F, $ positive $ \beta $ and the fact that $ (I-\calW) $ is positive semidefinite matrix. Given the current approximation $ \dbf^{k,\ell} \in \RnN, $ where $ \ell $ denotes the inner iteration counter $ 
\ell=0,1\ldots, $ for solving the linear system,  we can describe the next iteration of JOR method for solving (\ref{eq:Nsystem}) as follows. Denoting the Hessian $ H^k = [H_{ij}^k]= \nabla^2 \Phi_{\beta}(\xbf^k) \in \mathbf{R}^{nN \times nN}, H_{ii}^k = \nabla^2 f_i(\xbf^k_i) +\frac{1}{\beta }(1-w_{ii}) I_n,  H_{ij}^k = -\frac{1}{\beta } w_{ij}I_n, i\neq j, $ with the diagonal part of each $ H_{ii}^k $ denoted by $ D_{ii}^k, $ the gradient $ \gbf_k = (\gbf^k_1,\ldots, \gbf^k_N), \; \gbf^k_i =  \nabla f_i(\xbf^k_i) +\frac{1}{\beta}((1-w_{ii})\xbf^k_i - \sum_{j \in O_I} w_{ij}\xbf^k_j), i=1,\ldots,N $ and with the relaxation parameter $ \omega \in  \R$, that will be defined later on, the next iteration of JOR method for solving (\ref{eq:Nsystem}) is given by 
  \begin{equation}\label{eq:JORiter}
 \dbf^{k,\ell+1}_i = \dbf^{k,\ell}_i + \omega (D^k_{ii})^{-1}\lr{\gbf^k_i-\sum_{j=1}^N H^k_{ij}\dbf^{k,\ell}_j }.
 \end{equation}

 Performing enough iterations of (\ref{eq:JORiter}) we can get $ \dbf^k = \dbf^{k,\ell} $ such that  (\ref{eq:residual}) holds. It is easy to see that (\ref{eq:JORiter}) is distributed as each node $ i $ holds the blocks $ H_{ii}^k, H_{ij}^k, j \in O_i, $ i.e., each node holds the corresponding row $ H_i^k$ of the Hessian $ H^k, $ the local gradient component $ \gbf_i^k $ and the nodes need to share their approximate solutions $ \dbf_i^{k,\ell} $ only among neighbours. The method is converging for a suitable value of $ 0 < \omega < 2 \beta (1-\bar{w})/(M + 2 \beta), $ with $ \bar{w} = \max_{1\leq i\leq N} w_{ii}, $ see \cite{EFIX} for details. 
 From now on we will assume that $ \omega $ is chosen such that JOR method converges, with further implementation details  postponed to Section 7. 
 
 Notice that, in order for parameter $\omega$ to be properly set, 
 nodes need to know beforehand the global constants $\bar{w}$ and $M$.
  These two constants that corresponds to maxima of certain local nodes' quantities across the network can be precomputed by running beforehand two independent 
  algorithms for the maximum computation such as \cite{Johansson}.
  In addition, DINAS can alleviate the requirement for such knowledge if 
  JOR is replaced by the following iterative method:
    \begin{equation}\label{eqn-JOR-modified}
 \dbf^{k,\ell+1}_i = \left[ \nabla^2 f_i(\xbf^k_i) + (1/\beta) I\right]^{-1}
 \left( \sum_{j \in O_i} w_{ij} (1/\beta) \dbf^{k,\ell}_j +  \gbf^k_i\right),
 \end{equation}
 for $i=1,...,N$ in parallel. 
 The iterative method \eqref{eqn-JOR-modified} does not require 
 any beforehand knowledge of global parameters, at the cost of the 
 requirement for local Hessian inverse calculations.
  It is easy to see that \eqref{eqn-JOR-modified} 
  converges linearly, and there holds:
    \begin{equation*}\label{eqn-JOR-modified-rate}
 \| \dbf^{k,\ell+1} -  \dbf^{k,\star} \| 
 \leq \frac{1/\beta}{1/\beta + \mu} \,  \| \dbf^{k,\ell} -  \dbf^{k,\star} \| ,
 \end{equation*}  
 where 
 $ \dbf^{k,\star} =  
  \left[\nabla^2 \Phi_{\beta}(\xbf^k)\right]^{-1} \nabla \Phi_{\beta}(\xbf^k) $ 
  is the exact solution of \eqref{eq:Nsystem}.
  
 Having a distributed algorithm for solving the system of linear equations, let us explain the adaptive step size strategy employed here. The basic assumption is that the global constants $ \mu $ and $ L $ are not available. 
Classical line search method, without the global constants, is not applicable in the distributed settings as it requires the value of the global function at each node, which would be very costly in terms of communications. Furthermore, several values of global function value might be needed in each iteration to compute the line search step size. Thus we employ the procedure that  is governed by a sequence of parameters $ \{\gamma_k\} $ that are used to generate the step sizes,  adopting the reasoning for the Newton method from \cite{Polyak} to the case of distributed and inexact method. In each iteration we try the step size based on the current value of $ \gamma_k $ and check if such step size generates enough decrease in the gradient. If the decrease is large enough  the step size is accepted and we proceed to a new iteration. If not, the step size is rejected, the value of $ \gamma_k $ is reduced and a new step size is tried with the same direction. Thus the checking procedure is cheap as the direction stays the same and we will show that it eventually ends up with a step size that generates enough decrease. 
 
 The step size computation includes the infinity norm of the gradient in the previous iteration and  all nodes need this value. Therefore we use the algorithm from \cite{flooding} for exchange of information around the network. Recall that in \cite{flooding} the nodes exchange either local Hessians or their rank one approximations, while in DINAS they exchange only scalars. The algorithm is included here for the sake of completeness.  The distributed maximum computation can be performed exactly in a finite number of communication rounds, at most equal to the network diameter  \cite{Johansson}. However we used the algorithm from \cite{flooding} as this does not make an important difference for DINAS and makes numerical comparison presented in Section 7 easier to follow.

 \begin{alg}[DSF]\label{alg:DSF} $\ $\\
\textbf{Input:} Each node $ i: \{S_i\}_{i\in\N}$ scalar messages 
\begin{algorithmic}[1]
\STATE Set $I^0_i = \{S_i\}$
\FOR{$k=0,\dots,N-1$}
    \FOR{$j\in O_i$}
        \STATE take $S\in I^{k-1}_0$ such that $S$ was not received from node $j$ and not sent to node $j$ at the previous iterations
        \STATE send $S$ to node $j$
    \ENDFOR 
    \STATE update $I^{k}_i$ adding the messages received by the neighbors
\ENDFOR
\end{algorithmic}
\end{alg}
 Once the step is accepted the nodes run  DSF algorithm exchanging the infinity norms of local gradients to get the infinity norm of the aggregate gradient and proceed to the new iteration. The algorithm is stated in a node-wise fashion, to facilitate the understanding of the distributed flow although a condensed formulation will be used later on, for theoretical considerations.

\begin{alg}[DINAS]\label{alg:DINAS}$\ $\\
\textbf{Iteration $k$,  each node $i $ holds: $\eta_k,\  \xbf^k_i,\ \gbf^k_i  ,\ H_{i}^k,\ \gamma_k>0,\ \|\gbf^k\|_{\infty},\ q\in(0,1) $ }
\begin{algorithmic}[1]
\STATE All  nodes run distributed JOR iterations (\ref{eq:JORiter}) to compute $\{\dbf^k_i\}_{i=1}^N$ such that \begin{equation}\label{localcond}
  \|H^k_i  \dbf^k_i - \gbf^k_i\|_\infty\leq\eta_k\|\gbf^k\|_\infty,  \enspace \forall i=1,\dots,N
\end{equation}
\STATE All nodes compute the step size 
\begin{equation*}\label{adaptive_step}\alpha_k = \min \left\{1, \frac{1-\eta_k}{(1+\eta_k)^2}\gamma_k\frac{1}{\|\gbf^k\|_{\infty}}\right\}
\end{equation*}
\STATE Each node $i$ computes $\hat\xbf_i = \xbf^k_i-\alpha_k\dbf_i^k$
\STATE All nodes  share $\hat\xbf_i$ with the neighbors in $ O_i $
\STATE Each node $i$ computes $\hat \gbf_i = \nabla f_i(\hat\xbf_i) + \frac{1}{\beta}\lr{\hat\xbf_i -\sum_{j} w_{ij}\hat\xbf_j}$
\STATE All nodes run DSF (Algorithm \ref{alg:DSF}) with message at node $i$ given by $S_i = \|\hat \gbf_i\|_{\infty}$
\STATE Each node $i$ computes $\|\hat\gbf\|_{\infty} = \max_{j=1:N} S_j$
\STATE Each node $i$ performs the following check and conditional updates
 \IF {\begin{equation}\label{cond1}
 \alpha_k < 1 \mbox{ and }
\|\hat\gbf\|_\infty\leq \|\gbf^k\|_\infty - \frac{1}{2}\frac{(1-\eta_k)^2}{(1+\eta_k)^2}\gamma_k
 \end{equation} or  
 \begin{equation}\label{cond2}
 \alpha_k = 1 \mbox{ and }
\|\hat\gbf\|_\infty\leq\eta_k\|\gbf^k\|_\infty+\frac{1}{2\gamma_k}(1+\eta_k)^2\|\gbf^k\|_\infty^2
 \end{equation}}
  \STATE set $\gamma_{k+1} = \gamma_k$
  \STATE set $\xbf^{k+1}_i = \hat\xbf_i$,\ $\gbf^{k+1}_i = \hat \gbf_i,\ \|\gbf^{k+1}\|_{\infty} = \|\hat \gbf\|_{\infty} $
 \STATE define $H^{k+1}_i = \left(H^{k+1}_{i1},\dots,H^{k+1}_{iN}\right)\in\R^{n\times nN}$ with
  $$H^{k+1}_{ii} =
    \nabla^2f_i(\xbf^{k+1}_i)+ \frac{1}{\beta}(1-w_{ii})I_n, \;  H^{k+1}_{ij} = 
    -\frac{1}{\beta}w_{ij}I_n, \;  j\neq i
  $$
  \ELSE 
  \STATE set $\gamma_{k} = q\gamma_k$ and return to line 2
 \ENDIF
\end{algorithmic}
\end{alg}

Let us briefly comment the computational and computational costs of DINAS here, with further details postponed to Section 7. Application of JOR method for solving (\ref{eq:Nsystem}) implies that in each inner (JOR) iteration all nodes share their current approximations $ \dbf^{k,\ell}_i $ with neighbours.  Each iteration of JOR method includes only the inversion of local diagonal matrix and is hence rather cheap. Depending on the value of $ \eta_k $ the number of inner iterations can vary but the right hand side of (\ref{eq:residual}) ensures that initially we solve (\ref{eq:Nsystem}) rather loosely, with relative large residual and small computational effort while the accuracy requirements increase as the gradients gets smaller and we approach the solution. Therefore we completely avoid inversion of (possibly very dense) local Hessians
that is needed in NN method  \cite{NetworkNewton} and in one version of DAN, \cite{flooding}. Thus DINAS should be particularly effective if the dimension of (\ref{eq:minsum_newton}) is relatively large as the cost of JOR approximate solution  should be significantly smaller than the cost of exact solution (local Hessian inversion) in each iteration of Network Newton method. 

The condition (\ref{localcond}) does not need to be verified online if a  pre-defined number of iterations is used and we will estimate that number later on. Otherwise, if online verification is used, the nodes would need to run another distributed maximum procedure to see that each node satisfies the condition but such addition appears unnecessary.  Communication costs of exchanges needed to compute $ \|g(\xbf^k)\|_{\infty}  $ and $ \|\hat\gbf\|_\infty $ are not large as the nodes exchange only scalar values - the infinity norm of local gradient components. 

The step size computation at line 2 is performed by each node but they all compute the same value  so that step is well defined. The check and update defined at line 9 is again performed by all nodes using the same values and hence the result will be the same at each node. Therefore, all nodes go back to line 2 or all nodes update the approximate solution $ \xbf^k_i $ and the if loop at line 9 is well defined. In the next section we will prove that the if loop has finite termination.

Regarding the global knowledge requirements of system constants, to implement Step 6 (DSF or \cite{Johansson}), all nodes need to know beforehand an upper bound on the number of nodes $N$, or on network diameter. 
 Next, one can utilize \eqref{eqn-JOR-modified} with a single inner iteration and 
 $\eta=1/(1+\beta \mu)$, and global convergence of DINAS is ensured; see Theorem 
 4.1 ahead. Hence, for global convergence, nodes only need to know 
 beforehand an upper bound on $N$ and a lower bound on $\mu$. 
 Notice that $\mu$ can be calculated beforehand by running 
 a distributed minimum algorithm computation \cite{Johansson}
  with the initialization at node $i$ 
  given by its local cost's strong convexity constant.
   When compared with alternative second order solvers 
   such as NN \cite{NetworkNewton} and \cite{ErminWei}, 
   this is a significantly reduced prior knowledge requirement as these methods require the Lipschitz constans of Hessian  and 
    the gradient.  
   When DINAS employs additional tuning and knowledge of global constants, 
   as it is assumed by Theorems 4.2--4.4, then 
   stronger results in terms of convergence rates are ensured.

DINAS requires the scalar maximum-compute in Step 6 and synchronized activities of all nodes in Steps 6--14. It is interesting to compare DINAS with \cite{flooding} that requires a similar procedure. While \cite{flooding} utilizes such procedure over local Hessians ($O(n^2)$-sized local messages), DINAS requires Step 6 only for a scalar quantity. 
 It is also worth noting that, in view of Steps 6--14, DINAS is not easily generalizable to asynchronous settings as done in \cite{ErminWei}, or to unreliable 
 link/nodes settings although direction might be computed by asynchronous JOR method \cite{AF}.  However, for reliable and synchronized networks, 
 we emphasize that DINAS incurs drastic improvements both theoretically (quadratic convergence)
  and in experiments, exhibiting orders of magnitude faster rates when compared with existing alternatives.

\section{Convergence analysis of DINAS for personalized distributed optimization} 
Let us start the analysis proving that DINAS is well defined , i.e. proving that the  loop defined at line  9 has finite termination. 
\begin{lemma} \label{le:finite}
Let Assumptions \ref{ass:network} - \ref{ass:objetive} hold. Then the following statements hold:
\begin{enumerate}[i)]
    \item if $\gamma_k\leq \mu^2/L$ then one of  the condition at line 9 is satisfied; 
    \item the number of times $\gamma_k$ is reduced is bounded from above by
    $\log_{1/q}\lr{\gamma_0 L/\mu^2};$
    \item for every $k\in\N_0$ we have $\bar\gamma<\gamma_k\leq\gamma_0$, with $\bar \gamma = q\frac{\mu^2}{L}$;
    \item there exists $\bar m_1\in\N_0$ such that $\gamma_k = \gamma_{\bar m_1}$ for every $k\geq\bar m_1. $
\end{enumerate}
\end{lemma}
\begin{proof}
Since $f$ is two times continuously differentiable we have that for any $\alpha\in\R$ 
\begin{equation*}
\gbf(\xbf^k-\alpha \dbf^k) = \gbf(\xbf^k) - \alpha H^k\dbf^k - \int_0^1 \alpha \left(H(\xbf^k-s\alpha \dbf^k)- H^k\right)\dbf^k ds
\end{equation*}
and therefore, by  Assumption \ref{ass:objetive}, \eqref{eq:Nsystem} and \eqref{eq:residual}
\begin{equation*} 
\begin{aligned} \|\gbf(\xbf^k -\alpha \dbf^k)\|_\infty  & \leq \|\gbf^k - \alpha(\gbf^k+\rbf^k)\|_\infty  + \alpha \int_0^1 \|H(\xbf^k-s\alpha \dbf^k)- H^k\|_\infty\|\dbf^k\|_\infty ds \\  & \leq |1-\alpha|\|\gbf^k\|_\infty+ \alpha\|\rbf^k\|_\infty+ \alpha^2L\int_0^1 s\|\dbf^k\|_\infty ^2 ds \\
&\leq |1-\alpha|\|\gbf^k\|_\infty+ \alpha\eta_k\|\gbf^k\|_\infty+ \frac{\alpha^2L}{2}\|(H^k)^{-1}(\gbf^k+\rbf^k)\|_\infty ^2, 
\end{aligned}
\end{equation*}
and we get 
\begin{equation}\label{g_next}
 \|\gbf(\xbf^k -\alpha \dbf^k)\|_\infty \leq (|1-\alpha|+\eta_k\alpha)\|\gbf^k\|_\infty+ \frac{1}{2}\frac{L}{\mu^2} (1+\eta_k)^2 \alpha^2 \|\gbf^k\|_\infty ^2.
\end{equation}  
To prove the first statement we have to show that if $\gamma_k\leq\mu^2/L$ then either \eqref{cond1} or \eqref{cond2} hold.
From \eqref{g_next}, if $\alpha_k=1$ we have
$$\|\hat\gbf\|_\infty \leq\eta_k\|\gbf^k\|_\infty +\frac{1}{2}\frac{L}{\mu^2}(1+\eta_k)^2\|\gbf^k\|_\infty ^2.$$
Otherwise, for $\alpha_k<1$, we get
\begin{equation*}
    \begin{aligned}
    \|\hat\gbf\|_\infty &\leq \lr{1-\frac{\gamma_k(1-\eta_k)^2}{\|\gbf^k\|_\infty(1+\eta_k)^2}}\|\gbf^k\|_\infty 
    +\frac{L}{2 \mu^2}(1+\eta_k)^2\lr{\frac{\gamma_k(1-\eta_k)}{(1+\eta_k)^2 \|\gbf^k\|_\infty}}^2\|\gbf^k\|_\infty ^2
   \\ &
    \leq\|\gbf^k\|_\infty +\frac{\gamma_k (1-\eta_k)^2}{(1+\eta_k)^2}\left(\frac{L}{2 \mu^2}\gamma_k-1\right).
    \end{aligned}
\end{equation*} 
Given that $\gamma_k\leq  \mu^2/L$ the desired inequalities follow in both cases and we get $i).$\\
By definition of $\gamma_{k+1}$ (lines 10 and 13 in Algorithm DINAS), the sequence $\{\gamma_k\}$
is non increasing, the value of $\gamma_k$ is reduced only when neither \eqref{cond1} nor \eqref{cond2} are satisfied and in that case we decrease
$ \gamma_k $ by a fixed $ q \in (0,1). $ This, together with $i)$ implies $ii)$ and $iii).$ Since we proved that $\{\gamma_k\}$ is bounded from below, $iv)$ follows. 
\end{proof}
Lemma \ref{le:finite} implies in particular that for $k$ large enough $\gamma_k$ becomes constant. While by $iii)$ we know that $\gamma_{\bar m_1}\geq q \mu^2/L$, the Lemma does not state that $\gamma_k$ will eventually reach $q \mu^2/L$.

Notice that the iteration of DINAS can be written in a compact form as follows. Given $ \xbf^k $ and $ \dbf^k $ such that (\ref{eq:residual}) holds, we have $ \xbf^{k+1} = \xbf^k - \alpha_k \dbf^k $ where
\be \label{eq:alpha}
\alpha_k  = \min \left\{1, \frac{1-\eta_k}{(1+\eta_k)^2}\gamma_k\frac{1}{\|\gbf^k\|_{\infty}}\right\}.
\ee 
In the next theorem we prove convergence to the unique solution of \eqref{eq:penaltyref}. 

\begin{theorem} \label{th:main}
Assume that \ref{ass:network} - \ref{ass:objetive} hold, $\{\eta_k\} $ is a nonincreasing sequence such that $ 0\leq\eta_k\leq \bar\eta <1 $ and $ \gamma_0 > 0. $  Let $ \{\xbf^k\} $ be an arbitrary  sequence generated by DINAS.  Then 
\begin{enumerate}[i)]
\item there exists $\bar m_2\in\N_0$ such that $\alpha_k=1$ for every $k\geq \bar m_1+\bar m_2$, and  \begin{equation}\label{m2}\bar m_2 \leq\left\lceil \frac{1}{C}\lr{\|\gbf^{\bar m_1}\|_\infty  - \bar\gamma\frac{1-\bar\eta}{(1+\bar\eta)^2}}+1\right\rceil, \; C = q\frac{\mu^2}{L}\frac{(1-\bar\eta)^2}{(1+\bar\eta)^2};\end{equation}
\item $\lim_{k \to \infty} \|\gbf^k\|_\infty  = 0; $
\item $\lim_{k \to \infty} \xbf^k = \xbf^*$, where $ \xbf^* $ is the unique solution  of (\ref{eq:penaltyref}). 
\end{enumerate}
\begin{proof}
Let us first assume that at iteration $k$ we have step size 
\begin{equation}\label{alpha_small}
    \alpha_k= \frac{1-\eta_k}{(1+\eta_k)^2}\gamma_k\frac{1}{\|\gbf^k\|_\infty }<1.
\end{equation}
Then  \eqref{cond1} implies 
\begin{equation}\label{g_nex1}
    \|\gbf^{k+1}\|_\infty \leq \|\gbf^k\|_\infty -\frac{1}{2}\gamma_k\frac{(1-\eta_k)^2}{(1+\eta_k)^2}.
\end{equation}
By Lemma \ref{le:finite} we have $\gamma_k\geq q \mu^2/L$. Moreover, since $\frac{(1-\eta)^2}{(1+\eta)^2}$ is a decreasing function of $\eta$ for $\eta\in(0,1)$ and $\eta_k\leq\bar\eta<1$ we have that, for every $k$
$$\frac{(1-\eta_{k})^2}{(1+\eta_{k})^2}\geq \frac{(1-\bar\eta)^2}{(1+\bar\eta)^2}>0, \mbox{ and } \gamma_k\frac{(1-\eta_k)^2}{(1+\eta_k)^2}\geq q\frac{\mu^2}{L}\frac{(1-\bar\eta)^2}{(1+\bar\eta)^2}=:C>0$$
Replacing the last inequality in \eqref{g_nex1} we get 
\begin{equation}\label{g_nex1_2}
    \|\gbf^{k+1}\|_\infty \leq \|\gbf^k\|_\infty  - C/2
\end{equation}
for every iteration index $k$ such that $\alpha_k<1.$

Let us now consider the case where $\alpha_k=1. $ The by definition of $\alpha_k$ there follows
\begin{equation}\label{gstep}
    \|\gbf^k\|_\infty \leq\gamma_k\frac{1-\eta_k}{(1+\eta_k)^2}.
\end{equation}
From this inequality and \eqref{cond2}  we have
\begin{equation}\label{g_nex2}
\begin{aligned}
    \|\gbf^{k+1}\|_\infty &\leq\eta_k\|\gbf^k\|_\infty +\frac{1}{2\gamma_k}(1+\eta_k)^2\|\gbf^k\|_\infty ^2 \\& \leq \eta_k\|\gbf^k\|_\infty +\frac{1}{2}\left(1-\eta_k\right)\|\gbf^k\|_\infty =\frac12(1+\eta_k)\|\gbf^k\|_\infty .
\end{aligned}
\end{equation}
Since $\eta_k\leq\bar\eta<1$, for every $k$ such that $\alpha_k=1$ we have, with $\rho=\frac12(1+\bar\eta),$  
\begin{equation}\label{phase2}\|\gbf^{k+1}\|_\infty \leq\rho\|\gbf^k\|_\infty.  \end{equation}

Let $k>\bar m_1$. If $\alpha_k=1$, by \eqref{gstep}, \eqref{g_nex2}, $\eta_{k+1}\leq\eta_k$, and $\gamma_k = \gamma_{\bar m_1} = \gamma_{k+1}$ we have 
$$\|\gbf^{k+1}\|_\infty \leq\rho\|\gbf^k\|_\infty \leq\rho\gamma_k\frac{(1-\eta_k)}{(1-\eta_k)^2}\leq \rho\gamma_{k+1}\frac{(1-\eta_{k+1})}{(1-\eta_{k+1})^2},$$
which implies $\alpha_{k+1} = 1.$ Denote with $\bar m_2$ the smallest positive integer such that $\alpha_{\bar m_1+\bar m_2} = 1$.
This, together with \eqref{alpha_small} which holds for $ k = \bar m_1+\bar m_2-1$  and \eqref{g_nex1} implies 
\begin{equation*}\begin{aligned}
    \gamma_{\bar m_1+\bar m_2-1}\frac{1-\eta_{\bar m_1+\bar m_2-1}}{(1+\eta_{\bar m_1+\bar m_2-1})^2}&<\|\gbf^{\bar m_1+\bar m_2-1}\|_\infty \leq \|\gbf^{\bar m_1+\bar m_2-2}\|_\infty -C \\ & \leq
    \|\gbf^{\bar m_1}\|_\infty -(\bar m_2-1)C,
\end{aligned}
\end{equation*}
and thus 
\begin{equation*}
\begin{aligned}
\bar m_2 & < \frac{1}{C}\lr{\|\gbf^{\bar m_1}\|_\infty  - \gamma_{\bar m_1+\bar m_2-1}\frac{1-\eta_{\bar m_1+\bar m_2-1}}{(1+\eta_{\bar m_1+\bar m_2-1})^2}}+1 \\ & \leq
\frac{1}{C}\lr{\|\gbf^{\bar m_1}\|_\infty  - \bar\gamma\frac{1-\bar\eta}{(1+\bar\eta)^2}}+1.
\end{aligned}
\end{equation*}
Since we already proved that $\alpha_{k} = 1$ implies $\alpha_{k+1}=1$ for every $k\geq \bar m_2$,  $i)$ holds. 
Inequalities \eqref{g_nex1_2} and \eqref{phase2}, together with $i)$, imply part $ii)$ of the statement.\\
It remains to prove that the sequence of iterates $\{\xbf^k\}$ converges to the unique solution of (\ref{eq:penaltyref}).  For every $j\in\N$ we have, by \eqref{eq:Nsystem}, \eqref{eq:residual}, \eqref{phase2} and the bound on $\eta_k$
\begin{equation*}
    \begin{aligned}
    &\|\dbf^{\bar m_1+\bar m_2+j}\|_\infty \leq \|\lr{H^{\bar m_1+\bar m_2+j}}^{-1}(\gbf^{\bar m_1+\bar m_2+j}+\rbf^{\bar m_1+\bar m_2+j})\|_\infty  \\
    &\leq \frac{1}{\mu}(1+\bar\eta)\|\gbf^{\bar m_1+\bar m_2+j}\|_\infty \leq \frac{1}{\mu}(1+\bar\eta)\rho\|\gbf^{\bar m_1+\bar m_2+j-1}\|_\infty 
     \leq \frac{1}{\mu}(1+\bar\eta)\rho^j\|\gbf^{\bar m_1+\bar m_2}\|_\infty .
    \end{aligned}
\end{equation*}
Thus, for every $s\geq l\geq 0$
\begin{equation*}
    \begin{aligned}
   &\|\xbf^{\bar m_1+\bar m_2+s}-\xbf^{\bar m_1+\bar m_2+l}\|_\infty \leq\sum_{j=l}^{s-1}\|\dbf^{\bar m_1+\bar m_2+j}\|_\infty   \\ &\leq
   \frac{1}{\mu}(1+\bar\eta)\|\gbf^{\bar m_1+\bar m_2}\|_\infty \sum_{j=l}^{s-1}\rho^j =  \frac{1}{\mu}(1+\bar\eta)\|\gbf^{\bar m_1+\bar m_2}\|_\infty \frac{\rho^s-\rho^l}{1-\rho}.
    \end{aligned}
\end{equation*}
So,  $\{\xbf^k\}$ is a Cauchy sequence and therefore there exists $\bar\xbf = \lim_{k\rightarrow+\infty}\xbf^k.$ Since we already proved that $\lim_{k\rightarrow+\infty}\|\gbf^k\|_\infty =0$, we have that $\bar\xbf = \xbf^*. $ 
\end{proof}
\end{theorem}

\begin{remark}
Theorem \ref{th:main} shows that, like line search in the classical framework, the adaptive strategy employed by Algorithm \ref{alg:DINAS} ensures convergence to the solution for any initial guess (i.e. global convergence), and also that the full step $\alpha_k=1$ is accepted whenever  $\|\gbf^k\|_\infty $ is small enough. 
\end{remark}

\begin{remark}
    The method is convergent  whenever $\eta_k\leq\bar\eta<1.$ For suitable choices of the relaxation parameter the spectral radius of the iterative matrix of the JOR method is bounded away from 1, \cite{DMY}. Therefore Theorem \ref{th:main} hold if at each iteration of Algorithm \ref{alg:DINAS} the nodes perform only one iteration of JOR method and we have global convergence. The number of JOR iterations needed for (\ref{localcond}) is discussed later. 
\end{remark}

The forcing sequence $ \{\eta_k\} $ determines the rate of convergence as in the centralized optimization, \cite{INmethod}. We obtain linear, superlinear or quadratic convergence with a suitable choice of $ \eta_k $ as stated below. 

\begin{theorem} \label{th:rate}
Assume that \ref{ass:network} - \ref{ass:objetive} hold, $ \gamma_0 > 0 $ and let $ \{\xbf^k\} $ be a sequence generated by DINAS algorithm. If $ \{\eta_k\}$ is a nonincreasing sequence such that  $\eta_k\leq\bar\eta$ for $ \bar{\eta} $ small enough then $\{x^k\}$ converges linearly in norm $\|\cdot\|_{\infty}$. If 
$\eta_k \leq \eta\|\gbf^k\|_{\infty}^{\delta} $ for some $\eta\geq 0$ and $ k $ large enough,
then the convergence is superlinear for $\delta \in (0,1)$ and  quadratic  for $\delta=1$. 
\end{theorem}
\begin{proof}
We already proved in Theorem \ref{th:main} that $\|\gbf^k\|_\infty $ and $\xbf^k$ converge to 0 and $\xbf^*$ respectively. In the following, we always assume that $k\geq\bar m = \bar m_1+\bar m_2$, and hence $\alpha_k=1$ and $ \gamma_k = \gamma_{\bar{m}}. $ For $\delta = 0$, linear convergence of $\|\gbf^k\|_\infty $ follows directly from \eqref{phase2}. Let us consider the case $\delta>0$. For $k$ large enough, from \eqref{g_nex2} we have
\begin{equation*}
\begin{aligned}
    &\|\gbf^{k+1}\|_\infty \leq\eta_k\|\gbf^k\|_\infty +\frac{1}{2\gamma_k}(1+\eta_k)^2\|\gbf^k\|_\infty ^2  \\
    &\leq \eta\|\gbf^k\|_\infty ^{1+\delta}+\frac{1}{\gamma_k}\|\gbf^k\|_\infty ^2
    \leq \eta\|\gbf^k\|_\infty ^{1+\delta}+\frac{1}{\gamma_{\bar m}}\|\gbf^k\|_\infty ^2.
\end{aligned}
\end{equation*}
If $\delta = 1$ then
\begin{equation*}
\begin{aligned}
    \lim_{k\rightarrow+\infty}\frac{\|\gbf^{k+1}\|_\infty }{\|\gbf^k\|_\infty ^2}\leq \lim_{k\rightarrow+\infty}\frac{\eta\|\gbf^k\|_\infty ^2+\frac{1}{\gamma_{\bar m}}\|\gbf^k\|_\infty ^2}{\|\gbf^k\|_\infty ^2} = \eta+\frac{1}{\gamma_{\bar m}},
\end{aligned}
\end{equation*}
which ensures quadratic convergence. If $\delta\in(0,1)$, then   $\lim_{k \to \infty}  \|\gbf^k\|=0$ implies 
\begin{equation*}
\begin{aligned}
    \lim_{k\rightarrow+\infty}\frac{\|\gbf^{k+1}\|_\infty }{\|\gbf^k\|_\infty}&\leq \lim_{k\rightarrow+\infty}\frac{\eta\|\gbf^k\|_\infty ^{1+\delta}+\frac{1}{\gamma_{\bar m}}\|\gbf^k\|_\infty ^2}{\|\gbf^k\|_\infty } = 0 
\end{aligned}
\end{equation*}

Let us now consider the sequence $\xbf^k.$ For every $j\in\N$ we have 
\begin{equation}\label{diffsol}
\begin{aligned}
    \|\xbf^{\bar m+j+1}-\xbf^*\|_\infty  & = \|\xbf^{\bar m+j} -\xbf^* - \dbf^{\bar m+j}\|_\infty   \\ & = 
    \|\xbf^{\bar m+j} -\xbf^* - (H^{\bar m+j})^{-1}(\gbf^{\bar m+j}+\rbf^{\bar m+j})\|_\infty   \\ &=
    \|(H^{\bar m+j})^{-1}\|_\infty \|H^{\bar m+j}(\xbf^{\bar m+j} -\xbf^*) - (\gbf^{\bar m+j}+\rbf^{\bar m+j})\|_\infty   \\ & \leq
    \frac{1}{\mu}\|H^{\bar m+j}(\xbf^{\bar m+j} -\xbf^*) - (\gbf^{\bar m+j}-\gbf^*)\|_\infty  + \frac{1}{\mu}\|\rbf^{\bar m+j}\|_\infty .
\end{aligned}
\end{equation}
Since the objective function is twice continuously differentiable, we have 
\begin{equation*} 
\begin{aligned}
    &\|H^{\bar m+j}(\xbf^{\bar m+j} -\xbf^*) - (\gbf^{\bar m+j}-\gbf^*)\|_\infty \\ & = \left\| \int_0^1 (H^{\bar m+j} - H(\xbf^{\bar m+j}+s(\xbf^{\bar m+j}-\xbf^*))(\xbf^{\bar m+j}-\xbf^*) ds \right\|_\infty   \\ & \leq
    \int_0^1 sL\|\xbf^{\bar m+j}-\xbf^*)\|_\infty ^2 ds \leq \frac{L}{2}\|\xbf^{\bar m+j}-\xbf^*)\|_\infty ^2. 
\end{aligned}
\end{equation*}
Replacing this term in \eqref{diffsol} and using the bound on $\|\rbf^k\|_\infty $ we get
\begin{equation}\label{diffsol2}
\begin{aligned}
    &\|\xbf^{\bar m+j+1}-\xbf^*\|_\infty  \leq
    \frac{1}{\mu}\frac{L}{2}\|\xbf^{\bar m+j}-\xbf^*\|_\infty ^2 + \frac{1}{\mu}\eta_{\bar m+j}\|\gbf^{\bar m+j}\|_\infty .
\end{aligned}
\end{equation}
 By Assumption A3 we have
$$\|\gbf^k\|_\infty =\|\gbf^k - \gbf^*\|_\infty \leq M\|\xbf^k - \xbf^*\|_\infty.$$
Let us consider the case $\delta>0$ and let us notice that, since $\xbf^k$ converges to $\xbf^*$, we can assume $j$ is large enough so that $\|\xbf^{\bar m+j}-\xbf^*\|_\infty <1$. So, by definition of $\eta_k$ 
\begin{equation*}
\begin{aligned}
    \|\xbf^{\bar m+j+1}-\xbf^*\|_\infty  &\leq
    \frac{1}{\mu}\frac{L}{2}\|\xbf^{\bar m+j}-\xbf^*\|_\infty ^2 + \frac{1}{\mu}\eta\|\gbf^{\bar m+j}\|_\infty ^{1+\delta} \\ & \leq
    \frac{1}{\mu}\lr{\frac{L}{2} + M\eta}\|\xbf^{\bar m+j}-\xbf^*\|_\infty ^{1+\delta}
\end{aligned}
\end{equation*}
which proves superlinear and quadratic convergence for $\delta\in(0,1)$ and $\delta = 1,$ respectively.
For $\delta = 0$ and $\|\xbf^{\bar m+j}-\xbf^*\|_\infty \leq\varepsilon$ we have \eqref{diffsol2} and  $\eta_k\leq\bar\eta, $ and  
\begin{equation*}
\begin{aligned}
    &\|\xbf^{\bar m+j+1}-\xbf^*\|_\infty  \leq
    \frac{1}{\mu}\lr{\frac{L}{2}\varepsilon + M\bar\eta}\|\xbf^{\bar m+j}-\xbf^*\|_\infty 
\end{aligned}
\end{equation*}
For $\varepsilon,\bar\eta$ small enough we have that 
$\frac{1}{\mu}\lr{\frac{L}{2}\varepsilon + M\bar\eta}<1, $
which ensures linear convergence and concludes the proof. \end{proof}

The next statement gives the complexity result, i.e. we estimate the number of iterations needed to achieve $ \|\gbf_k\|_{\infty} \leq \varepsilon. $

\begin{theorem}
Assume that \ref{ass:network} - \ref{ass:objetive} hold, $\{\eta_k\} $ is a nonincreasing sequence of forcing parameters such that $ 0\leq\eta_k\leq \bar\eta <1 $ and $ \gamma_0 > 0. $  Let $ \{\xbf^k\} $ be an arbitrary sequence generated by DINAS.  The number of iterations necessary to reach $\|\gbf^k\|_{\infty}\leq\varepsilon$ for a given $\varepsilon>0$ is bounded from above by $$k_\varepsilon = \left\lceil\frac{\log\lr{\varepsilon^{-1}\|\gbf^0\|_\infty }}{\log(\hat\rho^{-1})}+1\right\rceil$$
with 
$$\hat\rho = \max\left\{\frac{(1+\bar\eta)}{2},\ 1-q\frac{\mu^2}{L}\frac{(1-\bar\eta)^2}{(1+\bar\eta)^2}\frac{1}{\|\gbf^0\|_\infty }\right\}<1.$$
\end{theorem}
\begin{proof}
Let us consider inequalities \eqref{g_nex1} and \eqref{phase2} derived in the proof of Theorem \ref{th:main}. For every index $k$ we have that if $\alpha_k=1$  then $\|\gbf^{k+1}\|_\infty \leq\rho\|\gbf^k\|_\infty  $
with $\rho = (1+\bar\eta)/2$. If $\alpha_k<1$ and $C = q\frac{\mu^2}{L}\frac{(1-\bar\eta)^2}{(1+\bar\eta)^2},$ then $$ 
    \|\gbf^{k+1}\|_\infty \leq \|\gbf^k\|_\infty -C/2 \leq \|\gbf^k\|_\infty \lr{1-\frac{C}{2\|\gbf^0\|_\infty }} =: \rho_2\|\gbf^k\|_\infty. $$
Since $\alpha_k<1$, the definition of $\alpha_k $ and inequalities  $\gamma_k\geq q\mu^2/L$ and $\eta_k\leq\bar\eta<1$ imply
$$       \|\gbf^k\|_\infty >\gamma_k\frac{1-\bar\eta_k}{(1+\bar\eta_k)^2}\geq q\frac{\mu^2}{L}\frac{1-\bar\eta}{(1+\bar\eta)^2}>C $$
and thus $\rho_2\in (0,1).$
That is, denoting with $\hat\rho = \max\{\rho,\rho_2\}<1$, we have 
$$\|\gbf^{k+1}\|_\infty \leq\hat\rho\|\gbf^k\|_\infty .$$
Let us denote with $k_\varepsilon$ the first iteration such that $\|\gbf^{k_\varepsilon}\|_\infty \leq\varepsilon.$ From the inequality above, we have
$$\varepsilon<\|\gbf^{ k_\varepsilon-1}\|_\infty \leq \hat\rho\|\gbf^{k_\varepsilon-2}\|_\infty \leq \hat\rho^{k_\varepsilon-1}\|\gbf^0\|_\infty ,$$
which implies
$$ \hat\rho^{k_\varepsilon-1}>\frac{\varepsilon}{\|\gbf^0\|_\infty }$$
and thus
$$k_\varepsilon-1<\log_{1/\hat\rho}\lr{\varepsilon^{-1}\|\gbf^0\|_\infty } = \frac{\log\lr{\varepsilon^{-1}\|\gbf^0\|_\infty }}{\log(\hat\rho^{-1})}, $$
which concludes the proof. 
\end{proof}

To complement the above result we give here an estimated number of inner (JOR) iterations needed to satisfy (\ref{localcond}). 
\begin{lemma}\label{lemma:inner_it}
Assume that at every outer iteration $ k $ the nodes run JOR method starting at $\dbf^{k,0}_i = 0$ for every $i=1,\dots,N, $ and let us denote with $M_k(\omega_k)$ the iterative matrix of JOR method. 
Then the number of JOR iteration performed by the nodes to satisfy the termination condition \eqref{localcond} is bounded from above by:
    $$\bar \ell_k = \left\lceil\frac{\ln\lr{\eta_k}}{\ln\lr{\|M_k(\omega_k)\|_{\infty}}}\right\rceil.$$
\begin{proof}
    For every $\ell\in\N$, we have 
    \begin{equation*}
        \begin{aligned}
            &\|H^k \dbf^{k,\ell} - \gbf^k\|_\infty =\|H^k\dbf^{k,\ell-1} + \omega_kH_kD^{-1}_k(\gbf^k - H_k\dbf^{k,\ell-1})- \gbf^k\|_\infty\\ &= 
            \|(I-\omega_kH_kD_k^{-1})(H_k\dbf^{k,\ell-1}-\gbf^k)\|_\infty\leq \|M_k(\omega_k)\|_\infty\|H\dbf^{k,\ell-1} - \gbf^k\|_\infty.
        \end{aligned}
    \end{equation*}
    Recursively applying this inequality and using the fact that $\dbf^{k,0} = 0$, we get the following bound for the residual at the $\ell$-th iteration of JOR method:
    \begin{equation*}
        \begin{aligned}
            &\|H^k \dbf^{k,\ell} - \gbf^k\|_\infty \leq \|M_k(\omega_k)\|_\infty^\ell \|H\dbf^{k,0} - \gbf^k\|_\infty =  \|M_k(\omega_k)\|_\infty^{\ell}\|\gbf^k\|_\infty.
        \end{aligned}
    \end{equation*}
    If $ \ell \geq \left\lceil\frac{\ln\lr{\eta_k}}{\ln\lr{\|M_k(\omega_k)\|_{\infty}}}\right\rceil$
    then 
    $\|H^k \dbf^{k,\ell} - \gbf^k\|_\infty\leq \eta_k\|\gbf^k\|_\infty $
 and therefore the statement is proved.
\end{proof}
\end{lemma}

Assume that for every iteration $k$ the forcing term $\eta_k$ is given by $\eta_k = \eta\|\gbf^k\|^{\delta}$. Theorem \ref{th:rate} then ensures local linear, superlinear or quadratic convergence 
depending on the value of $ \delta$ and at each iteration the following inequality holds
\begin{equation*}
    \begin{aligned}
        &\|\gbf^{k+1}\|_\infty \leq\eta_k\|\gbf^k\|_\infty+\frac{1}{\gamma_{\bar m}}\|\gbf^k\|_\infty ^2 =\nu_k \|\gbf^k\|_\infty,
    \end{aligned}
\end{equation*}
where 
\begin{equation} \label{crate} 
\nu_k = \lr{\eta_k+\frac{1}{\gamma_{\bar m}}\|\gbf^k\|_\infty}.  
\end{equation} 
Therefore we can estimate the rate of convergence with respect to the communication rounds. Given that the step size $ \alpha_k =1 $ is eventually accepted for $ k $ large enough, the total number of communications rounds per outer iteration is $ \ell_k$ communications for the JOR method, plus the sharing of the local Newtonian directions, i.e. the total number of communications is governed by $ \ell_k. $ The following statement claims that the rate of convergence with respect to the communication rounds is always linear, with the convergence factor depending on $ \eta_k, $ i.e., on $ \delta. $

To be more precise, 
we introduce the following quantity:
 $$ \rho_{\delta} := \lim_{k \to \infty} \nu_k^{1/\bar{\ell_k}} <1. $$
 The above limit  exists as shown ahead.
  Quantity $\rho_{\delta}$ may be 
  seen as an asymptotic convergence factor with respect to the number of 
  communication rounds. 
  Indeed, given that, at the outer iteration $k$, the number of 
  communication rounds is governed by $\ell_k$,
   it follows that the multiplicative factor of the error decay 
   per a single inner iteration (communication round)
    at iteration $k$ equals $\nu_k^{1/\bar{\ell_k}}$. 
    Hence, taking the limit of the latter quantity 
    as $k \rightarrow \infty$ gives the asymptotic 
    (geometric--multiplicative) convergence factor 
    with respect to the number of communication rounds.

\begin{theorem} \label{th:comrate}
    Let $ \delta \in [0,1] $ and $ \eta_k = \eta \|g_k\|_{\infty}^\delta $ for $ \eta > 0$ small enough and assume that at each iteration the JOR parameter $\omega_k$ is chosen in such a way that $\|M_k(\omega_k)\|_{\infty}\leq\sigma<1$. Then 
    the rate of convergence with respect to the communications rounds of DINAS method is linear, i.e. 
    $ \rho_\delta<1. $

\begin{proof}
    We will distinguish two cases, depending on the value of $\delta.$
    Let us first consider  $\delta = 0$. Since in this case $\eta_k = \eta$ for every $k$, by Lemma \ref{lemma:inner_it} we have that the number of inner iterations is bounded from above by 
    $$\bar \ell_k = \left\lceil\frac{\ln\lr{\eta}}{\ln\lr{\|M_k(\omega_k)\|_{\infty}}}\right\rceil\leq\left\lceil\frac{\ln\lr{\eta}}{\ln(\sigma)}\right\rceil =: \bar\ell.$$
    From \eqref{phase2} we have $\nu_k = \rho$ for every iteration index $k$ and 
    $$\rho_0 = \lim_{k \to \infty} \nu_k^{1/\bar{\ell_k}} \leq \rho^\frac{\ln(\sigma)}{\ln\lr{\eta}}  <1,$$
    which proves the thesis for $\delta = 0.$\\
    Let us now consider the case $\delta\in(0,1].$ 
    From \eqref{crate}, the definition of $\eta_k$, and the fact that $\{\|\gbf^k\|\}$ is a decreasing sequence that tends to 0, we have that for $k$ large enough $$\nu_k\leq \lr{\eta+\frac{1}{\gamma_{\bar m}}}\|\gbf^k\|^\delta_\infty.$$
    Therefore
    \begin{equation*}\begin{aligned}
        \rho_\delta &= \lim_{k \to \infty} \nu_k^{1/\bar{\ell_k}} 
        = \lim_{k \to \infty} \lr{ \lr{\eta+\frac{1}{\gamma_{\bar m}}}\|\gbf^k\|^\delta_\infty}^{\frac{\ln(\sigma)}{\ln(\eta\|\gbf^k\|^\delta_\infty)}} \\&=
        \lim_{k \to \infty} \exp\lr{\ln(\sigma)\frac{\ln\lr{ \lr{\eta+\frac{1}{\gamma_{\bar m}}}\|\gbf^k\|^\delta_\infty}}{\ln\lr{\eta\|\gbf^k\|_\infty}}} = e^{\ln(\sigma)} = \sigma.
    \end{aligned}
    \end{equation*}
    Since $\sigma\in(0,1)$ this implies $\rho_\delta = \sigma<1$,
    and the proof is complete.
\end{proof}
\end{theorem}

\begin{remark} Theorem \ref{th:comrate}, jointly with Theorem \ref{th:rate},  
establish a quadratic (respectively, superlinear) convergence rate 
 for $\delta=1$ (respectively, $\delta \in (0,1)$) with respect to 
 outer iterations (number of local gradient and local Hessian evaluations) and a linear convergence rate with respect to the number of communication rounds. This is a strict improvement 
 with respect to existing results like, e.g., \cite{ANewton}, 
 that only establish a linear rate with respect to the number of local gradient and Hessian evaluations 
 and the number of communication rounds. More precisely, \cite{ANewton} provides a bounded range of iterations during which the ``convergence rate'' corresponds to a ``quadratic'' regime. In contrast, we establish here the 
 quadratic or superlinear rate in the asymptotic sense. 
 \end{remark}

\begin{remark}
    For $\delta>0$, the asymptotic linear rate 
with respect to the number of communication rounds equals $\sigma$, i.e., 
it matches the rate of the inner iteration JOR method. In other words, the rate is the same as if 
the JOR method was run independently for solving \eqref{eq:Nsystem} with $\mathbf{r}^k=0$. 
 Intuitively, for $\delta>0$, the outer iteration process exhibits at least a superlinear rate, 
 and therefore the outer iteration process has no effect (asymptotically) 
 on the communication-wise convergence rate. 
 \end{remark}
 
\begin{remark} The convergence factor $\sigma$ for the communication-wise linear convergence rate established here (the $\delta>0$ case) is significantly improved over the existing 
 results like \cite{ANewton}. The rate here is improved as it corresponds to 
 the inner JOR process and is not constrained (deteriorated) by the requirement on the sufficiently small Newton direction step size, as it is the case in \cite{ANewton}. 
To further illustrate this, when \eqref{eqn-JOR-modified} inner solver is used, 
 the final linear convergence factor communication-wise equals 
 $1/(1+ \beta \mu)$, and it is hence independent of 
 the local functions' condition numbers, network topology, or local Hessian Lipschitz constants.
\end{remark}

\section{Analysis of inexact centralized Newton method with Polyak's adaptive step size}
The adaptive step size we use in DINAS can be traced back to the adaptive step sizes proposed in \cite{Polyak} for the Newton method. The convergence results we obtain are in fact derived generalizing the reasoning presented there, taking into account both the distributed computational framework and approximate Newton search direction. Assuming that $ N = 1, $ i.e., considering the classical problem of solving (\ref{eq:minsum_newton})  on a single computational node we can state the adaptive step size method for Inexact method with the same analysis as already presented. Thus if we consider the method 
stated for (\ref{eq:minsum_newton}) in the centralized computation framework we get the algorithm bellow. Here $ \|\cdot\| $ is an arbitrary norm. 

\begin{alg}[DINASC]\label{alg:DINASC}$\ $\\
\textbf{Iteration $k:$ $\eta_k,\  y^k,\  \gamma_k>0,\ q\in(0,1) $ }
\begin{algorithmic}[1]
\STATE Compute $ d^k $ such that \begin{equation}\label{localcondc}
    \|\nabla^2 f(y^k) d^k - \nabla f(y^k)\| \leq \eta_k \|\nabla f(y^k)\|
\end{equation}
\STATE Compute the step size 
\begin{equation}\label{adaptive_stepc}\alpha_k = \min \left\{1, \frac{1-\eta_k}{(1+\eta_k)^2}\gamma_k\frac{1}{\|\nabla f(y^k)\|}\right\}
\end{equation}
\STATE Compute the trial point  $ \hat{y} = y^k - \alpha_k  \nabla f(y^k) $
 \IF {\begin{equation}\label{cond1c}
 \alpha_k < 1 \mbox{ and } 
\|\nabla f(\hat{y})\| \leq \|\nabla (f(y^k) \| - \frac{1}{2}\frac{(1-\eta_k)^2}{(1+\eta_k)^2}\gamma_k
 \end{equation} or  \begin{equation}\label{cond2c}
 \alpha_k = 1 \mbox{ and }
\|\nabla f(\hat{y})\|\leq\eta_k\|\nabla f(y^k) \|+\frac{1}{2\gamma_k}(1+\eta_k)^2\|\ \nabla f(y^k)\|^2
 \end{equation}}
  \STATE set $\gamma_{k+1} = \gamma_k$ and  $y^{k+1} = \hat{y}$
  \ELSE 
  \STATE set $\gamma_{k} = q\gamma_k$ and return to line 2
 \ENDIF
\end{algorithmic}
\end{alg}

The statements already proved for the distributed case clearly imply the global convergence of the sequence $ \{y^k\}$ as stated below. 

\begin{theorem} \label{th:centralized}
Assume that \ref{ass:network} - \ref{ass:objetive} hold and that the iterative sequence $ \{y^k\} $ is generated by Algorithm DINASC. 
Then $ \lim_k y^k = y^* $ and the rate of convergence is governed by the forcing sequence $ \{\eta_k\} $ as in Theorem \ref{th:rate}. 
\end{theorem}

Assuming that the constants $ \mu $ and $ L $ are available we get the following statement for  Inexact Newton methods with arbitrary linear solver.   
\begin{corollary} \label{co:exact}
Assume that \ref{ass:network} - \ref{ass:objetive} hold and that the iterative sequence $ \{y^k\} $ is generated as 
$ y^{k+1} = y^k - \hat\alpha_k d^k, $
where $ d^k $ satisfies (\ref{localcondc}),  and 
\be \label{eq:exactalpha}
\hat\alpha_k = \min \left\{1, \frac{1-\eta_k}{(1+\eta_k)^2} \frac{\mu^2}{L} \frac{1}{\|\nabla f(y^k)\|}\right\}. 
\ee
Then $ \lim_k y^k = y^* $ and the rate of convergence depends on the forcing sequence $ \{\eta_k\} $ as in Theorem \ref{th:rate}. 
\end{corollary}

\begin{proof}
The step size employed in DINASC algorithm reduces to $\hat\alpha_k$ whenever $\gamma_k = \mu^2/L$, while from part $i)$ of Lemma \ref{le:finite} we have that for this choice of $\gamma_k$ either condition \eqref{cond1c} or \eqref{cond2c} is always satisfied. That is, for the considered sequence, we have 
$$\|\nabla f(y^{k+1})\| \leq \|\nabla f(y^k)\| - \frac{1}{2}\frac{\mu^2}{L}\frac{(1-\eta_k)^2}{(1+\eta_k)^2}$$
for every $k$ such that $\hat\alpha_k<1$, and 
$$\|\nabla f(y^{k+1})\|\leq\eta_k\|\nabla f(y^k)\|+\frac{1}{2}\frac{L}{\mu^2}(1+\eta_k)^2\|\nabla f(y^k)\|^2$$
for every $k$ such that $\hat\alpha_k=1$. The thesis  follows directly from the analysis of DINAS.
\end{proof}

\begin{remark}
The previous corollary provides a choice of the step size that is accepted at all iterations. However, compared to $\hat\alpha_k$, the adaptive step size \eqref{adaptive_stepc}, i.e.  \eqref{eq:alpha} in DINAS, presents several advantages. First of all, the definition of $\hat\alpha_k$ involves the regularity constants $L$ and $\mu$, which are generally not known. Moreover, even when the constants are known,  $\hat\alpha_k$ could be very small, especially if in the initial iterations the gradient is large. 
The numerical experience so far implies that for a reasonable value of $ \gamma_0 $ we have a rather small number of rejections in Step 4 (Step 9 of DINAS)  and the step size is mostly accepted although $  \gamma_k >\mu^2/L$. Notice that when $\gamma_k>\mu^2/L$, the right hand sides of inequalities \eqref{cond1c} and \eqref{cond2c} are smaller than their equivalent for $\gamma_k=\mu^2/L.$ That is,  the adaptive step size generates  the decrease in the gradient is larger than the decrease induced by $\hat\alpha_k.$ 
\end{remark}

\section{Convergence analysis for DINAS: Consensus optimization}

Let us finally address the issue of convergence towards solutions of (\ref{eq:minsum_newton}) and (\ref{eq:penaltyref}). As already explained the solution of (\ref{eq:penaltyref}) is an approximate solution of (\ref{eq:minsum_newton}), and each local component $ \xbf_i^* $ of the penalty problem solution $ \xbf^* $ is in the $ \beta$- neighbourhood of  $ \ybf^* $ - the solution of (\ref{eq:minsum_newton}). So, one might naturally consider a sequence of penalty problem 
\be \label{eq:penaltyseq}
\min \Phi_{\beta_s} = F(x) + \frac{1}{2 \beta_s}\xbf^T (I - \calW) \xbf 
\ee
with a decreasing sequence $ \beta_{s} $ to reach the solution of (\ref{eq:minsum_newton}) with arbitrary precision and mimic the so-called exact methods discussed in Introduction. Thus we can solve each penalty method determined by $ \beta_{s} $ by DINAS up to a certain precision and use  a warm start strategy (taking the solution of the problem with $ \beta_{s} $ as the starting point for the problem with $ \beta_{s+1}$) to generate the solution of (\ref{eq:minsum_newton}). Naturally the exit criterion for each penalty problem and the way of decreasing $ \beta_{s} $ determine the properties of such sequence, in particular the rate of convergence of DINAS should be in line with the rate of decreasing $ \beta_{s}, $ to avoid oversolving of each penalty problem. 

This approach is similar to the penalty decomposition methods  in the centralized framework where one transforms the original problems, possibly with difficult constraints, into a sequence of problems that are easier to solve. The main difference is due to the change of computational framework - while in the penalty decomposition methods we distinguish between difficult and easy constraints, \cite {KL}, here we take exactly the same problem and change the value of $ \beta_s,$ to achieve suitable proximity of the solution (\eqref{eq:penaltyseq} to the solution of \eqref{eq:minsum_newton}. We demonstrate the efficiency of this approach, stated below as Algorithm SDINAS, in the next Section. 

\begin{alg}[SDINAS]\label{alg:penaltyDINAS}$\ $\\
\textbf{Input:} $\varepsilon_0,\beta_0>0,\ \hat\xbf^0\in\RnN, \theta\in(0,1).$
\begin{algorithmic}[1]
\FOR{$s=1,2,\dots$}
\STATE use DINAS starting at $\hat\xbf^{s-1}$ to find $\hat\xbf^s$ such that 
$\|\Phi_{\beta_s}(\hat\xbf^s)\|_\infty \leq\varepsilon_s $
\STATE set $\beta_{s+1} = \theta \beta_s$
\STATE set $\varepsilon_{s+1} = \theta\varepsilon_s$
\ENDFOR
\end{algorithmic}
\end{alg}

\begin{remark}
Different choices could  be made at lines 3 and 4 for the update of the penalty parameter $\beta_s$ and the tolerance $\varepsilon_s$. The fixed decrease proposed here is suitable for the distributed case as it does not require any additional communication among the nodes but the convergence theorem holds for more general   $\{\beta_s\}$ and $\{\varepsilon_s\}.$
\end{remark}

The following theorem shows that every accumulation point of the sequence generated by Algorithm \ref{alg:penaltyDINAS} is the solution of \eqref{eq:minsum_constrW}. Notice that the matrix $I-\calW$ is singular and thus the LICQ condition does not hold. Therefore we need to prove that an iterative sequence defined by Algorithm \ref{alg:penaltyDINAS} converges to the solution of \eqref{eq:penaltyref},  similarly to  \cite[Theorem 3.1]{EFIX}. The issue with LICQ is another property that is common with penalty decomposition method, \cite{KL}, where the constraint qualification might not hold and one needs to define alternative optimality conditions. 

\begin{theorem}
Let Assumptions \ref{ass:network} - \ref{ass:objetive} hold 
and let $\{\hat\xbf^s\}$ be a  sequence  such that 
$$\|\nabla \Phi_{\beta_s}(\hat\xbf^s)\|_\infty \leq\varepsilon_s.$$
If $\lim_{s\rightarrow+\infty}\beta_s = \lim_{s\rightarrow+\infty}\varepsilon_s = 0$, then every accumulation point of  $\{\hat\xbf^s\}$ satisfies the sufficient second order optimality conditions for problem \eqref{eq:minsum_constrW}.
\begin{proof}
Let $\bar\xbf$ be an accumulation point of $\{\hat\xbf^s\}$ and let $\calK_1\subseteq\N_0$ be an infinite subset such that $\lim_{k\in\calK_1}\hat\xbf^s = \bar\xbf.$
By definition of $\hat\xbf^s$ and $\Phi_{\beta_s}$, we have 
\begin{equation*}
    \begin{aligned}
        \varepsilon_{\beta_s}&\geq\|\nabla\Phi_{\beta_s}(\hat\xbf^s)\|_\infty
        \geq
        \frac{1}{\beta_{\beta_s}}\|(I-\calW)\hat\xbf^s\|_\infty-\|\nabla F(\hat\xbf^s)\|_\infty ,
    \end{aligned}
\end{equation*}
which implies 
\begin{equation}\label{aux1}
\|(I-\calW)\hat\xbf^s\|_\infty \leq \beta_s(\varepsilon_s+\|\nabla F(\hat\xbf^s)\|_\infty ).\end{equation}
Since $\nabla F$ is bounded over $\{\hat\xbf^s\}_{\calK_1}$, and $\beta_s$ tends to zero,
we get 

$$\|(I-\calW)\bar\xbf\|_\infty  \leq \lim_{k\in\calK_1}\beta_s(\varepsilon_s+\|\nabla F(\hat\xbf^s)\|_\infty ) = 0,$$
which also implies that $(I-\calW)^{1/2}\bar\xbf = 0$ and therefore $\bar\xbf$ is a feasible point for \eqref{eq:minsum_constrW}.
Let us now define for every $s\in\N_0$ the vectors $$\vbf_s = \frac{1}{\beta_s}(I-\calW)^{1/2}\hat\xbf^s,
 \; \zbf_s = \frac{1}{\beta_s}(I-\calW)\hat\xbf^s.$$
We will prove that $\{\vbf^s\}_{s\in\calK_1}$ is bounded. Let $I-\calW = U\Lambda U\tr$ be the eigendecomposition of $I-\calW$, with $\Lambda = \dia(\lambda_1,\dots,\lambda_{nN})$. From \eqref{aux1} we have that $\{\zbf^s\}_{s\in\calK_1}$ is bounded. That is, there exists $Z\in\R$ such that $\|\zbf^s\|\leq Z$ for every $s\in\calK_1$. Since $U$ is an orthogonal matrix, by definition of $\zbf$ we have
\begin{equation*}
        Z \geq\|\zbf^s\|\geq\|U\tr\zbf^s\| = \frac{1}{\beta_s}\|U\tr U\Lambda U\tr\hat\xbf^s\| = \frac{1}{\beta_s}\|\Lambda U\tr\hat\xbf^s\|=\frac{1}{\beta_s}\lr{\sum_{i=1}^{nN}\lambda_i^2 (U\tr\hat\xbf^s)^2_i }^{1/2}
\end{equation*}
Since $\lambda_i\geq0$ for every $i$, this implies that  $ \{\frac{1}{\beta_s}\lambda_i^2 (U\tr\xbf^s)^2_i\}_{s\in\calK_1}$ is bounded for every $i$ and therefore $ \{\frac{1}{\beta_s}\lambda_i (U\tr\xbf^s)^2_i\}_{s\in\calK_1}$ is also bounded. By definition of $\vbf^s$ we get
\begin{equation*}
    \begin{aligned}
        \frac{1}{\beta_s}\lr{\sum_{i=1}^{nN}\lambda_i (U\tr\xbf^s)_i }^{1/2} &=  \frac{1}{\beta_s}\|\Lambda^{1/2} U\tr\xbf\|=\|U\tr\vbf^s\|.
    \end{aligned}
\end{equation*}
The above equality implies that $\{\vbf^s\}_{s\in\calK_1}$ is bounded and therefore there exists $\bar\vbf\in\RnN$ and $\calK_2\subseteq\calK_1,$ an infinite subset such that
$\lim_{k\in\calK_2}\vbf^s = \bar\vbf$. By definition of $\hat\xbf^s$, $\Phi_s$, and $\vbf^s$ we have 
\begin{equation*}
    \begin{aligned}
        \varepsilon_s&\geq\|\nabla\Phi_s(\hat\xbf^s)\|_\infty  = \left\|\nabla F(\hat\xbf^s)+\frac{1}{\beta_s}(I-\calW)\hat\xbf^s\right\|_\infty    =
        \left\|\nabla F(\hat\xbf^s)+(I-\calW)^{1/2}\vbf^s\right\|_\infty .
    \end{aligned}
\end{equation*}
Taking the limit for $s\in\calK_2$ we get 
$$\left\|\nabla F(\bar\xbf)+(I-\calW)^{1/2}\bar\vbf\right\|_\infty  = 0,$$ 
and thus $\bar\xbf$ is  satisfies the KKT conditions for \eqref{eq:minsum_constrW}.
Denoting with $\cal L$ the Lagrangian function of problem \eqref{eq:minsum_constrW}, by Assumption \ref{ass:objetive} we have that $\nabla^2_{\xbf\xbf}{\cal L}(\bar\xbf,\bar\vbf)$ is positive definite, and therefore we get the thesis.
\end{proof}
\end{theorem}

\section{Numerical Results}\label{sec:NEWTON4}

We now present a set of numerical results to investigate the behavior of DINAS when solving \eqref{eq:penalty-intro} and  \eqref{eq:minsum_newton} and how it compares with relevant methods from the literature.  

\subsection{Numerical results for distributed personalized optimization} 
Given that the choice of the forcing terms  forcing terms ${\eta_k}$ influences the performance of the method , we begin by numerically veryfing theoretical results on the convergence rate. Consider the problem of minimizing a logistic loss function with $l_2$ regularization. That is, given $\{\textbf a_j\}_{j=1}^m\subset\R^n,\ \{b_j\}_{j=1}^m\subset\{-1,1\}, \rho>0$, the objective function $f$ is defined as 
\begin{equation}\label{logreg1}
f(\ybf) = \sum_{j=1}^m\ln\left(1+\exp(-b_j\textbf{a}_j\tr \ybf)\right) + \frac12\rho\|\ybf\|^2_2
\end{equation}
We set $n = 100$, $m=1000$ and assume that node $i$ holds   $\{\textbf a_j\}_{j\in{\cal R}_i}$, $\{b_j\}_{j\in{\cal R}_i}$ for  $\mathcal R_i = \{(i-1) 100+1,\ldots,100i\}$
For every $j=1,\dots,m$ the components of $\textbf{a}_j$ are independent and uniformly drawn from $(0,1),$ while $b_j$ takes value $1$ or $-1$ with equal probability, while the regularization parameter is $\rho = 0.01m.$ The underlying communication network is defined as a random geometric graph with communication radius $\sqrt{N^{-1}\ln(N)}$, and the consensus matrix $W$ as the Metropolis matrix \cite{metropolis}. 
To evaluate the methods, we define the per-iteration total cost of each method as the sum of the computational cost plus the communication traffic multiplied by a scaling constant $r$, \cite{cost}. That is, 
\begin{equation}\label{totalcost}
    \text{total cost} = \text{computation} + r\cdot\text{communication}
\end{equation}
The computational cost is expressed in terms of scalar operations, while the communication traffic is the total number of scalar quantities shared by all nodes. The scaling factor $r$ is introduced to reflect the fact that the time necessary to share a variable between two nodes compared with the time necessary to execute scalar computations depends on many factors of technical nature, such as the kind of computational nodes that form the network and the technology they use to communicate, that are beyond the purpose of these experiments.  

\begin{figure}[h]
\centering
    \begin{subfigure}[b]{0.40\textwidth}
    \centering
        \includegraphics[width = 0.98\textwidth]{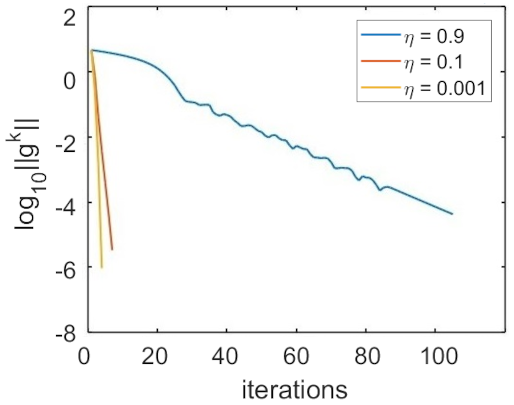}
 \caption{$\delta = 0$}\label{fig:NT0_1}
    \end{subfigure}
    \begin{subfigure}[b]{0.40\textwidth}
    \centering
        \includegraphics[width = \textwidth]{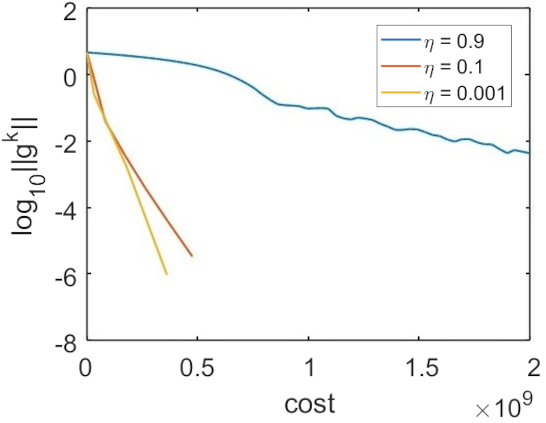}
 \caption{$\delta = 0$}\label{fig:NT0_2}
    \end{subfigure}

        \begin{subfigure}[b]{0.40\textwidth}
    \centering
        \includegraphics[width = \textwidth]{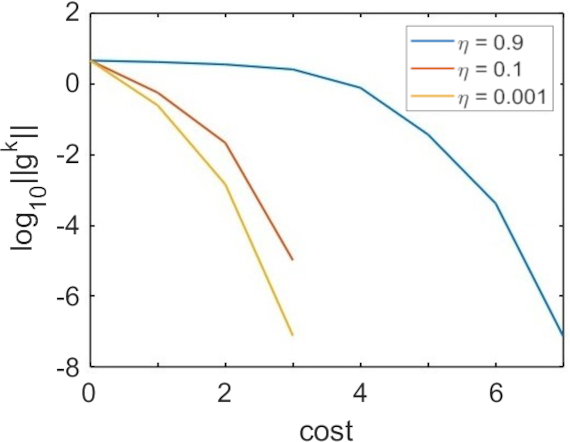}
 \caption{$\delta = 1$}\label{fig:NT0_3}
    \end{subfigure}
    \begin{subfigure}[b]{0.40\textwidth}
    \centering
        \includegraphics[width = \textwidth]{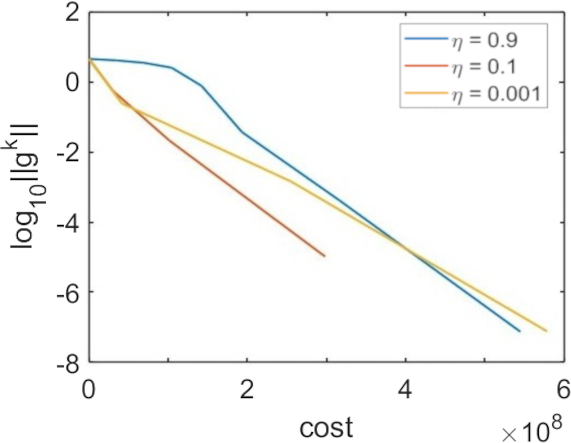}
 \caption{$\delta = 1$}\label{fig:NT0_4}
    \end{subfigure}
       \caption{Choice of the forcing terms, Logistic Regression }\label{fig:NT0}
\end{figure}

Given  $f$ in (\ref{logreg1}) as explained above and $\beta = 0.1$, the personalized optimization problem in the form of  \eqref{eq:penaltyref} is solved with DINAS algorithm for different choices of the sequence of forcing terms $\{\eta_k\}, $ defined as 
$$\eta_k = \min\{\eta, \eta\|\gbf^k\|_\infty ^\delta\}, $$ 
for $\delta=0,1$ and $\eta = 0.9,0.1,0.001$.
All nodes start with initial guess $\xbf^0_i = 0\in\Rn$ and the execution terminates when $\|\nabla\Phi_\beta(\xbf^k)\|\leq 10^{-5}.$ For all the methods we define $\gamma_0 = 1.$ In Figure \ref{fig:NT0} we plot the results for the six methods given by the different combinations of $ \delta $ and $ \eta.$ In Figure \ref{fig:NT0_1}, \ref{fig:NT0_2} we consider the case $\delta = 0$ (that is, $\eta_k=\eta$ for all $k$), while in Figure \ref{fig:NT0_3},\ref{fig:NT0_4} we have $\delta=1$. In each subfigure we plot the value of $\log_{10}(\|\gbf^k\|)$ versus iterations (Figure \ref{fig:NT0_1}, \ref{fig:NT0_3}) and cost (\ref{fig:NT0_2}, \ref{fig:NT0_4}), with scaling factor $r=1$. 
Figures \ref{fig:NT0_1}, \ref{fig:NT0_3} confirm the results stated in Theorem \ref{th:rate}: the sequence $\|\gbf^k\|$ is linearly decreasing for all the considered choices of  $\eta$, while for $\delta=1$ the convergence is locally quadratic. 
For both values of $\delta$ the number of iterations required by the methods to arrive at termination depends directly on the choice of the forcing term: smaller values of $\eta$ ensure the stopping criterion is satisfied in a smaller number of iterations. However, for $\delta=1$ we notice that, when compared in terms of overall cost, the method with the smallest value of $\eta$ performs worse than the other two. For $\delta=0$ the comparison among the methods for the cost gives the same result as that in terms of iterations. The results for different values of the cost scaling factor $r$ are completely analogous and are therefore omitted here.

\subsection{Comparison with Exact Methods for consensus optimization}
We compare DINAS with NN \cite{NetworkNewton}, DAN and DAN-LA \cite{flooding}, Newton Tracking \cite{Ntrack}, DIGing \cite{diging} and EXTRA \cite{extra}.  The proposed method DINAS is designed to solve the penalty formulation of the problem and  therefore, in order to minimize \eqref{eq:minsum_newton}, we apply Algorithm \ref{alg:penaltyDINAS} with $\beta_0 = 0.1$, $\beta_{s+1}  = 0.1 \beta_s$ and $\varepsilon_s = 0.01\beta_s$. For NN we proceed analogously, replacing DINAS in line 2 with Network Newton. All other methods are the so-called exact methods, and therefore can be applied directly to minimize $f$. We take $\gamma_0 = 1$, $\delta = 0$ and $\eta = 0.9$ in Algorithm \ref{alg:DINAS}, i.e., we consider linearly convergent DINAS, while for all other methods the step sizes are computed as in the respective papers. For DAN-LA the constants are set as in \cite{flooding}. In particular, we consider four different values of the parameter $c = M,10M,100M,1000M$  and denote them as DAN LR1,DAN LR2, DAN LR3, DAN LR4 in the figures bellow. 

First, we consider a logistic regression problem with the same parameters as in the previous test. The exact solution $\ybf^*$ of \eqref{logreg1} is computed by running the classical gradient method with tolerance $10^{-8}$ on the norm of the gradient. As in \cite{NetworkNewton}, the methods are evaluated considering the average squared relative error, defined as 
$$e_k = \frac{1}{N}\sum_{i=1}^N\frac{\|\xbf^k_i-\xbf^*\|^2}{\|\xbf^*\|^2}$$
where $\xbf^* = (\ybf^*,\dots,\ybf^*)\tr\in\RnN.$
For all methods the initial guess is $\xbf^0_i=0$ at every node, which yields $e_0 = 1$, and the execution is terminated when $e^k\leq 10^{-4}.$ We consider the same combined measure of computational cost and communication defined in \eqref{totalcost}, with scaling factor $r=0.1,1,10$ and plot the results in Figure \ref{fig:NT1}.

\begin{figure}[h]
\centering
    \begin{subfigure}{0.327\textwidth}
    \centering
        \includegraphics[width = \textwidth]{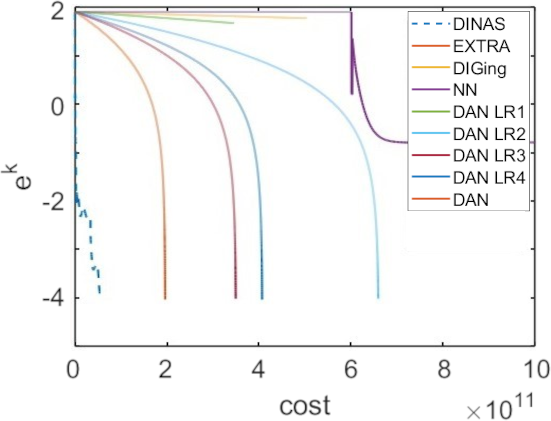}
 \caption{$r=0.1$}\label{fig:NT1_1}
    \end{subfigure}
    \begin{subfigure}{0.327\textwidth}
    \centering
        \includegraphics[width = \textwidth]{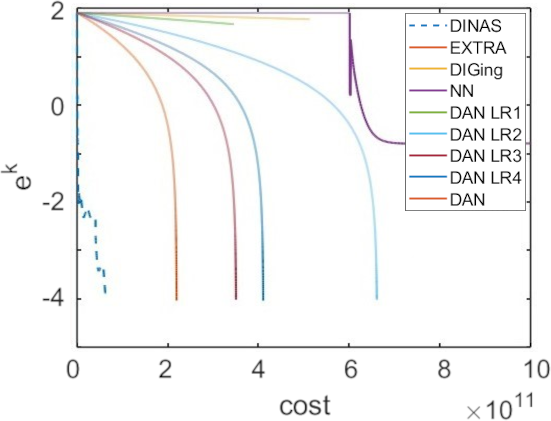}
 \caption{$r=1$}\label{fig:NT1_2}
    \end{subfigure}
        \centering
        \begin{subfigure}{0.327\textwidth}
      \centering
        \includegraphics[width = \textwidth]{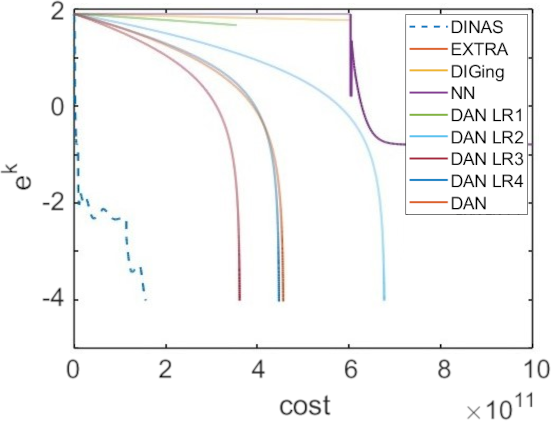}
 \caption{$r=10$}\label{fig:NT1_3}
    \end{subfigure}
   
    \caption{Total cost, Logistic Regression }\label{fig:NT1}
\end{figure}
One can see that for all values of $r$ DINAS outperforms all the other methods. NN, DIGing and EXTRA all work with fixed step sizes that, in order to ensure global convergence of the methods, need to be very small. Despite the fact that each iteration of DIGing and EXTRA is very cheap compared to an iteration of DINAS, this is not enough to compensate the fact that both these methods require a large number of iterations to arrive at termination.  DAN and DAN-LA methods  use an adaptive step size that depends on the constants of the problem $\mu$ and $L$ and on $1/\|\nabla \Phi^k\|_\infty$ in such a way that the full step size is accepted when the solution is approached. In fact, we can clearly see from the plots that all these methods reach a quadratic phase where $e_t$ decreases very quickly. However, the per-iteration cost of these methods is, in general, significantly higher than the cost of DINAS. DAN method requires all the local Hessians $\nabla^2 f_i(\xbf^k)$ to be shared among all the nodes at each iteration. While using Algorithm \ref{alg:DSF} this can be done in a finite number of rounds of communications, the overall communication traffic is large as it scales quadratically with both the dimension $n$ of the problem and the number of nodes $N$. DAN-LA avoids the communication of matrices by computing and sharing the rank-1 approximations of the local Hessians. While this reduces significantly the communication traffic of the method, it increases the computational cost, as two eigenvalues and one eigenvector need to be computed by every node at all iterations, and the number of iterations, since the direction is computed using an approximation of the true Hessian. Overall, this leads to a larger per-iteration cost than DINAS. Since $\gamma_0=1$ and it only decreases when the conditions \eqref{cond1},\eqref{cond2} do not hold, we have that $\alpha_k$ in DINAS is relatively large compared to the fixed step sizes employed by the other methods that we considered. The per-iteration cost of DINAS is largely dominated by the cost of JOR that we use to compute the direction $\dbf^k$. Since the method is run with $\eta_k$ large, and $\dbf^{k-1}$ is used as initial guess at the next iteration, a small number of JOR iteration is needed on average to satisfy \eqref{eq:residual}, which makes the overall computational and communication traffic of DINAS small compared to DAN and DAN-LA.

The  logistic regression problem is also solved with  Voice rehabilitation dataset - LSVT,\cite{lsvt}. The dataset is made of $m=126$ points with $n=309$ features, and the datapoints are distributed among $N=30$ nodes on a random network generated as above. The results for different values of $r$ are presented in Figure \ref{fig:LSVT}, and they are completely analogous to those obtained for the synthetic dataset in Figure \ref{fig:NT1}. 
\begin{figure}[h]
\centering
    \begin{subfigure}[b]{0.327\textwidth}
    \centering
        \includegraphics[width = \textwidth]{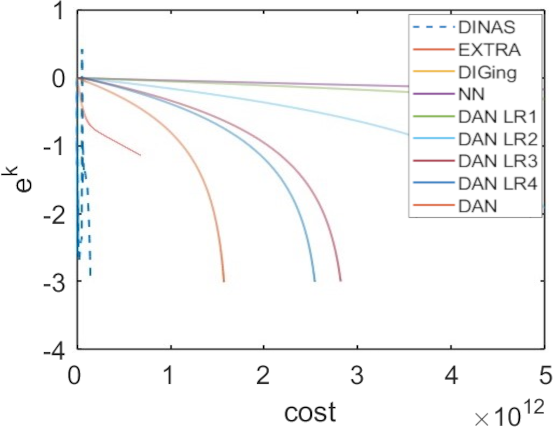}
 \caption{$r=0.1$}
    \end{subfigure}
    \begin{subfigure}{0.327\textwidth}
    \centering
        \includegraphics[width = \textwidth]{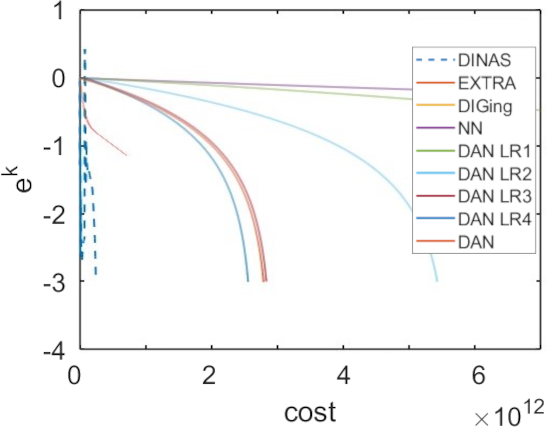}
 \caption{$r=1$}
    \end{subfigure}
       \centering
        \begin{subfigure}{0.327\textwidth}
      \centering
        \includegraphics[width = \textwidth]{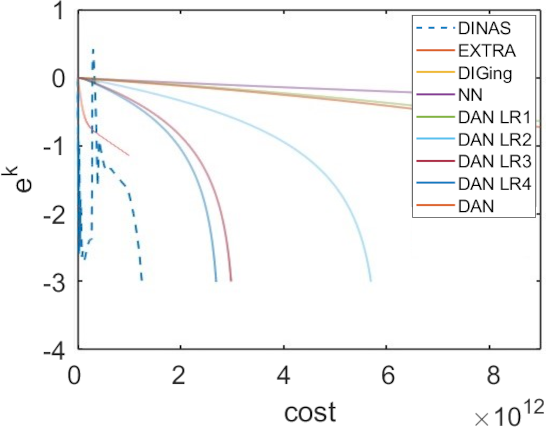}
 \caption{$r=10$}
    \end{subfigure}
   
    \caption{Total cost, Logistic Regression, LSVT dataset }\label{fig:LSVT}
\end{figure}

To investigate the influence of conditional number we consider a quadratic problem defined as  
\begin{equation}\label{minquad}
    f(\ybf) = \sum_{i=1}^N f_i(\ybf),\enspace f_i(\ybf) = \ybf\tr A_i\ybf + \ybf\tr\bbf_i 
\end{equation}
with $A_i\in\Rnn$, $b\in\Rn$ for every $i=1,\dots,N.$ We take $n=100$ and $N=10$ and we generate $A_i, b_i$ as follows.
Given $0<\lambda_{\min}<\lambda_{\max}$, we define the diagonal matrix $D_i = \dia(\lambda^i_1,\dots,\lambda^i_n)$ 
where the scalars $\lambda^i_j$ are independent and uniformly sampled in $[\lambda_{\min},\lambda_{\max}].$ 
Given a randomly generated orthogonal matrix $P_i\in\Rnn$ we define $A_i = P_iD_iP_i\tr.$ For every $i=1,\dots,n$ 
the components of $\bbf_i$ are independent and from the uniform distribution in $[0,1].$ Fixing  $\lambda_{\min} = 0.1$ different problems of the form \eqref{minquad} with increasing values of $\lambda_{\max}$ are considered. 
For each problem the exact solution $\ybf^*$, the initial guess and the termination condition are all set as in the previous test. The same combined measure of the cost, with scaling factor $r$ is used.  All methods are run with step sizes from the respective papers, while for NN we use step size equal to 1, as suggested in \cite{NetworkNewton} for quadratic problems. In Figure \ref{fig:quad} we plot the obtained results for $\lambda_{\max} = 1,10,100$ and $r=0.1,10$.

\begin{figure}[h]
\centering
    \begin{subfigure}[b]{0.327\textwidth}
    \centering
        \includegraphics[width = \textwidth]{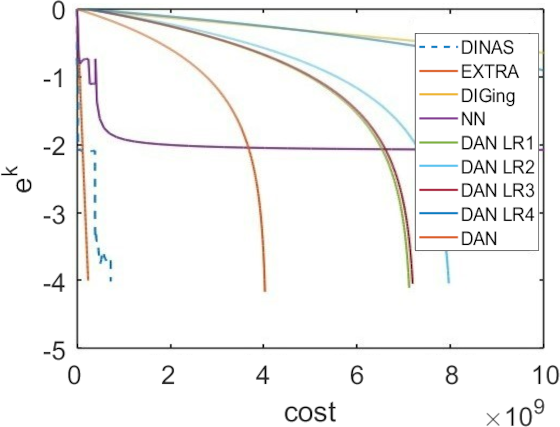}
 \caption{$\lambda_{\max}=1,\ r=0.1$}
    \end{subfigure}
    \begin{subfigure}{0.327\textwidth}
    \centering
        \includegraphics[width = \textwidth]{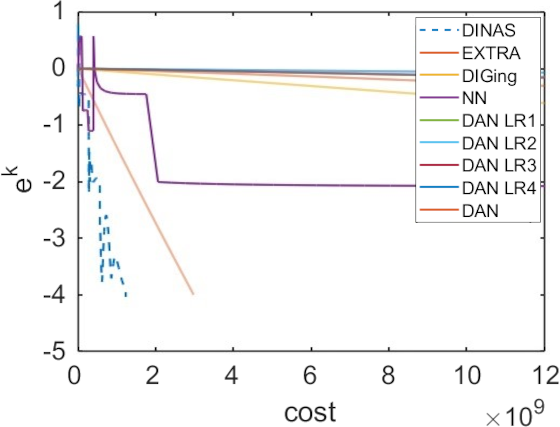}
 \caption{$\lambda_{\max}=10,\ r=0.1$}
    \end{subfigure}
       \centering
        \begin{subfigure}{0.327\textwidth}
      \centering
        \includegraphics[width = \textwidth]{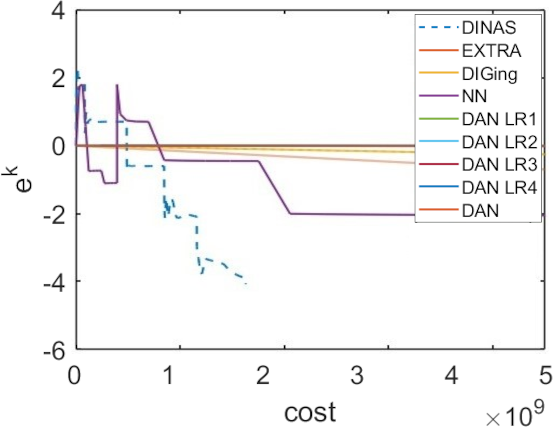}
 \caption{{$\lambda_{\max}=100,\ r=0.1$}}
    \end{subfigure}

    \begin{subfigure}[b]{0.327\textwidth}
    \centering
        \includegraphics[width = \textwidth]{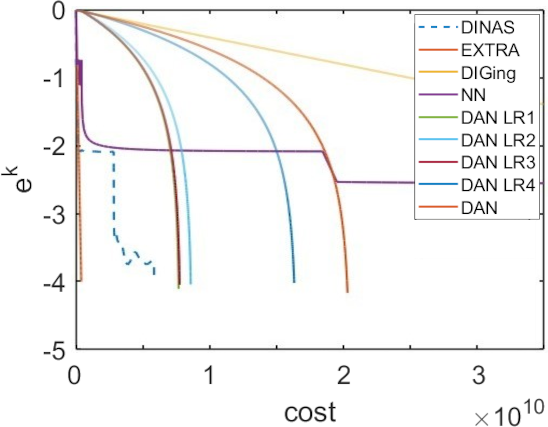}
 \caption{$\lambda_{\max}=1,\ r=10$}
    \end{subfigure}
    \begin{subfigure}{0.327\textwidth}
    \centering
        \includegraphics[width = \textwidth]{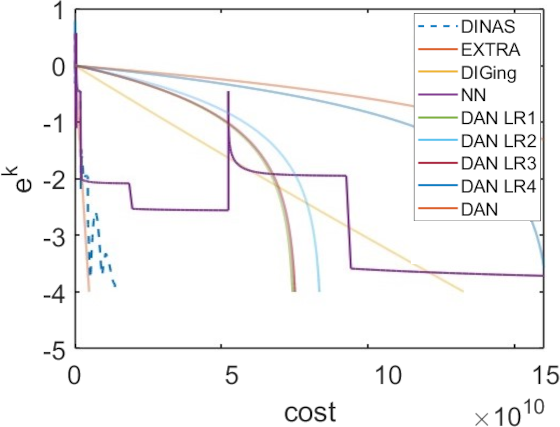}
 \caption{$\lambda_{\max}=10,\ r=10$}
    \end{subfigure}
       \centering
        \begin{subfigure}{0.327\textwidth}
      \centering
        \includegraphics[width = \textwidth]{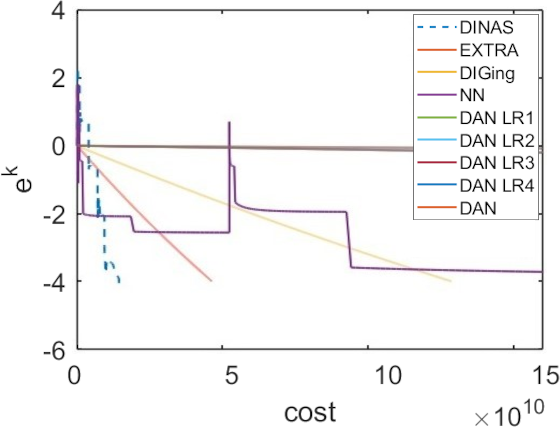}
 \caption{$\lambda_{\max}=100,\ r=10$}
    \end{subfigure}
   
    \caption{Total cost, quadratic problem }\label{fig:quad}

\end{figure}

For this set of problems the advantages of DINAS, compared to the other considered methods, become more evident as $\lambda_{\max}$ increases. When $\lambda_{\max}$ is larger, the Lipschitz constant of the problem also increases and therefore the step sizes that ensure convergence of DIGing and EXTRA become progressively smaller. In fact we can see that EXTRA outperforms the proposed method for $\lambda_{\max} = 1$ when the cost is computed with $r=0.1$ and for $\lambda_{\max} = 10$ when $r=10$, but DINAS becomes more efficient for larger values of $\lambda_{\max}.$ Regarding DAN and DAN-LA, what we noticed for the previous test also holds here. Moreover, their step size depends on the ratio $\mu^2/L$ which, for large values of $\lambda_{\max}$ causes the step size to be small for many iterations. While NN uses the full step size in this test,  its performance is in general more influenced by the condition number of the problem than that of DINAS. Moreover, while the per-iteration communication traffic of NN is fixed and generally lower than that of DINAS, the computational cost is typically larger, as at each iteration every node has to solve multiple linear systems of size $n$, exactly. Finally, we notice that for all the considered values of $\lambda_{\max}$ the comparison between DINAS and the other method is better for $r = 0.1$ which is a direct consequence of assigning different weight to the communication traffic when computing the overall cost.

\section{Conclusions}
The results presented here extend the classical theory of Inexact Newton methods to the distributed framework in the following aspects. An adaptive (large) step size selection protocol is proposed that yields global convergence. When solving personalized distributed optimization problems, the rate of convergence is governed by the forcing terms as in the classical case, yielding linear, superlinear or quadratic convergence with respect to computational cost. The rate of convergence with respect to the number of communication rounds is linear. The step sizes are adaptive, as in \cite{Polyak} for the Newton method, and they can be computed in a distributed way with a minimized required knowledge of global constants beforehand. For distributed consensus optimization, exact global convergence is achieved. The advantages of the proposed DINAS method in terms of computational and communication costs with respect to the state-of-the-art methods are demonstrated through several numerical examples, including large- scale and ill-conditioned problems. Finally, a consequence of the analysis for the distributed case is also convergence theory for a centralized setting, wherein adaptive step sizes and an inexact Newton method with arbitrary linear solvers is analyzed, hence extending the results in \cite{Polyak} to inexact Newton steps.

\bibliographystyle{plain}
\bibliography{references}
\end{document}